\documentclass[12pt, a4paper]{article}

\usepackage{amsmath} % For advanced math environments (e.g., align)
\usepackage{bm}
\usepackage{tabularx}
\usepackage{booktabs}
\usepackage{ragged2e}
\usepackage{threeparttable}
\usepackage{amssymb} % For common math symbols
\usepackage{amsthm}  % For theorem-like environments
\usepackage{graphicx} % For including figures
\usepackage[left=2.5cm,right=2.5cm,top=2.5cm,bottom=2.5cm]{geometry} % For margins
\usepackage{hyperref} % For clickable links and cross-references
\usepackage{authblk} % For multiple authors and affiliations

\usepackage{siunitx}
\usepackage{caption}

\usepackage{algorithm}
\usepackage{algpseudocode}
\usepackage{tikz}
\usepackage[noabbrev]{cleveref}
\usepackage{makecell}
\usepackage{multirow}

\newtheorem{definition}{Definition}
\newtheorem{fact}{Fact}
\newtheorem{remark}{Remark}
\newtheorem{corollary}{Corollary}
\newtheorem{theorem}{Theorem}
\newtheorem{lemma}{Lemma}
\newtheorem{proposition}{Proposition}

\usepackage[acronym]{glossaries}
\newacronym{cpwl}{CPWL}{continuous piecewise linear}
\newacronym{milp}{MILP}{mixed-integer linear programming}
\newacronym{dc}{DC}{difference-of-convex}
\newacronym{nn}{NN}{neural network}

\title{Tightening the mixed integer linear formulation for the piecewise linear approximation in general dimensions}

\author[1]{Quentin Ploussard}
\author[2]{Xiang Li}
\author[1]{Matija Pavičević}

\affil[1]{Argonne National Laboratory, Lemont, IL, USA}
\affil[2]{Zhejiang University, Hangzhou, China}

\date{\today}

\begin{document}

\maketitle % This command generates the title, authors, and date.

\begin{abstract}
This paper addresses the problem of tightening the mixed-integer linear programming (MILP) formulation for continuous piecewise linear (CPWL) approximations of data sets in arbitrary dimensions. The MILP formulation leverages the difference-of-convex (DC) representation of CPWL functions. We introduce the concept of well-behaved CPWL interpolations and demonstrate that any CPWL interpolation of a data set has a well-behaved version. This result is critical to tighten the MILP problem. We present six different strategies to tighten the problem, which include fixing the values of some variables, introducing additional constraints, identifying small big-M parameter values and applying tighter variable bounds. These methods leverage key aspects of the DC representation and the inherent structure of well-behaved CPWL interpolations. Experimental results demonstrate that specific combinations of these tightening strategies lead to significant improvement in solution times, especially for tightening strategies that consider well-behaved CPWL solutions.
\end{abstract}

% Keywords are not strictly required for arXiv but are a good practice.

% Uncomment if you want to include a table of contents.
% \tableofcontents
% \newpage

%-------------------------------------------------------------
% Main Body of the Document
% ------------------------------------------------------------
\section{Introduction}
\label{sec:Intro}
A \gls{cpwl} function is a continuous function defined on a compact set which can be partitioned into a finite set of affine domains whithin which the function is affine. 
\Gls{cpwl} functions are a fundamental tool in mathematics, engineering, and data science. This is due to their ability to approximate any continuous nonlinear function defined on a compact set to an arbitrary level of precision using a finite number of linear pieces (\cite{huang_relu_2020}). \gls{cpwl} functions are also an essential tool when fitting a data set for which the analytical expression of the nonlinear relationship is unknown (\cite{rebennack_piecewise_2020}). In optimization models, using \gls{cpwl} functions to represent complex nonlinear relationships enables the use of mature \gls{milp} solvers, which are often significantly faster than nonlinear solvers (\cite{vielma_mixed-integer_2010,vielma_mixed_2015}).

Not only can \gls{cpwl} functions be represented in \gls{milp} models, the optimal \gls{cpwl} fitting problem itself can be expressed as a \gls{milp} problem (\cite{toriello_fitting_2012}). The objective function may vary depending on the application. For instance, given a fixed number of affine functions, one may wish to determine the \gls{cpwl} approximation that minimizes an approximation error metric, \textit{e.g.}, the maximum or average error (\cite{rebennack_piecewise_2020}). Alternatively, one may wish to identify the \gls{cpwl} approximation with the minimal number of affine pieces (\cite{kong_derivation_2020}). 

\gls{milp} formulations for the \gls{cpwl} fitting problem have been extensively investigated for univariate functions and one-dimensional data sets. \cite{frenzen_number_2010} studied optimal partitioning of the range of the univariate function to be approximated and provided a quantitative result regarding how the number of segments depends on the function and the approximation error. \cite{rebennack_piecewise_2020} developed a finitely convergent method to generate \gls{cpwl} approximations for one-dimensional data sets or continuous univariate functions identifying the minimal number of breakpoints required to meet an acceptable approximation error. Similarly, \cite{kong_derivation_2020} introduced a novel method of enforcing continuity at breakpoints through a set of linear constraints. The two formulations are compared in detail in \cite{warwicker_comparison_2022}, and it is shown that the formulation from \cite{rebennack_piecewise_2020} is generally faster in practice. The work from \cite{warwicker_generating_2023} extends the work from \cite{warwicker_comparison_2022} and uses a Benders decomposition approach to improve the solution time of the \gls{milp} problem. \cite{ploussard_piecewise_2024} introduce a tightening method that significantly shorten the \gls{milp} solution of the \gls{cpwl} fitting problem when hierarchically minimizing the number of linear pieces and approximation error. 

For bivariate functions and two-dimensional data sets, most existing methods constrain the affine domains of the \gls{cpwl} approximation to be simplices. For example, \cite{dambrosio_piecewise_2010} partition the \gls{cpwl} function domain into rectangles that are later split into triangles along predefined diagonals. \cite{toriello_fitting_2012} use a similar method but select a set of splitting diagonals that minimize the approximation error. \cite{rebennack_continuous_2015} use a different triangulation strategy that iteratively partitions the triangles covering the approximation domain into smaller ones until a target approximation error is reached. \cite{duguet_piecewise_2022} propose a similar iterative refinement, but without imposing the affine domains to be simplices.  

The literature addressing the optimal \gls{cpwl} fitting problem for datasets of higher dimensions is significantly more limited (\cite{misener_piecewise-linear_2010, rebennack_continuous_2015}). Besides, most of the existing methods essentially extend the simplex-based partitioning strategy used for two dimensions, which result in computational challenges due to the exponential growth of number of possible simplex-based partitions as the data dimension increases (\cite{hughes_simplexity_1996}).  
\cite{kazda_nonconvex_2021} introduced the first \gls{milp} formulation able to identify the optimal \gls{cpwl} approximation of data sets of any dimension for a given maximum approximation error. They achieve this by leveraging the \gls{dc} representation of \gls{cpwl} functions (\cite{kripfganz_piecewise_1987}) which allows for the affine domains partitioning the \gls{cpwl} function to be implicitly defined by the difference of the pointwise maxima of two sets of affine functions. The strength of the \gls{dc} representation lies in its ability to describe any \gls{cpwl} function with affine domains of arbitrary shape. However, the computational cost of the \gls{milp} problem rapidly increases with the dimension of the data set, the number of data points, and the number of affine functions.

Deep learning techniques offer an alternative way to identify \gls{cpwl} approximations (\cite{huang_relu_2020}). In fact, the function represented by a \gls{nn} using ReLU activation function is a \gls{cpwl} function and the \gls{nn} loss function represents the \gls{cpwl} approximation error (\cite{devore_neural_2021}). To the best of our knowledge, the only existing methods to identify \gls{cpwl} approximations of datasets in multiple dimensions without explicitly specifying the partitioning of affine domains are the ReLU \gls{nn} approach and the \gls{dc} \gls{milp} approach introduced by \cite{kazda_nonconvex_2021}. Training a \gls{nn} to identify a good \gls{cpwl} approximation is significantly faster than solving a \gls{milp} problem, contributing to \gls{nn} growing popularity in recent years. However, the \gls{milp} approach provides three key advantages over the ReLU \gls{nn} approach:
\begin{itemize}
    \item The \gls{nn} approach uses a gradient-based method that identifies local optimas whereas the \gls{milp} approach is guaranteed to yield the \gls{cpwl} function with the smallest approximation error.
    \item Contrary to the \gls{milp} approach, the \gls{nn} approach does not allow for the explicit enforcement of a target maximum approximation error.
    \item Due to the complex relationship between the structure (width and depth) of a \gls{nn} and the number of affine pieces of the \gls{cpwl} functions they can model, it is difficult for \gls{nn}s to identify good \gls{cpwl} approximations with a manageable number of affine pieces (\cite{chen_improved_2023}). Conversely, the \gls{milp} approach allows the user to identify \gls{cpwl} functions with a predefined, or minimal, number of linear pieces. This is a critical feature when these \gls{cpwl} functions are subsequently embedded in \gls{milp} formulations to model nonlinear relationships, where the number of binary variables required to model \gls{cpwl} functions is typically proportional to the number of affine pieces that compose them. In other words, while a \gls{nn} can quickly identify a good-quality \gls{cpwl} approximation of a nonlinear relationship, the resulting approximation may render the \gls{milp} formulation in which it is embedded computationally intractable.
\end{itemize}

To address the high computational cost of solving the \gls{milp} problem, authors in \cite{kazda_linear_2024} introduce an iterative LP approach that identifies a valid \gls{cpwl} approximation given a target maximum error. However, the identified \gls{cpwl} solution is not necessarily optimal. A priori, identifying the global optimum of the \gls{cpwl} fitting problem can only be achieved by solving a \gls{milp} problem. Therefore, it is critical to improve the \gls{milp} formulation of the \gls{cpwl} fitting problem to reduce its computational cost. The method described in \cite{ploussard_piecewise_2024} achieves a tightening of the \gls{milp} problem that significantly improves the solution time, but the method only applies to one-dimensional data sets. This raises the need to develop \gls{milp} tightening methods that can be extended to data sets of any dimensions.

This paper aims to address this research gap. To achieve this, we introduce a critical class of \gls{cpwl} interpolations called ``well-behaved'' and demonstrate that any \gls{cpwl} interpolation has a well-behaved version. This important result is then leveraged to identify six tightening strategies that can be combined to significantly reduce the solution time of the \gls{milp} \gls{cpwl} fitting problem.

The five main contributions of this paper are the following. First, we formalize the concept of well-behaved \gls{cpwl} interpolations, a class of \gls{cpwl} interpolations in which each affine piece interpolates a number of points greater than the dimension of the interpolated data set. Second, we demonstrate that any \gls{cpwl} interpolation of a data set has a well-behaved version, which serves as a key theoretical foundation for our tightening method. Third, we leverage this result to introduce six tightening strategies for the \gls{milp} formulation of the \gls{cpwl} fitting problem. Fourth, we analyze the theoretical impact of each tightening strategy together with the time complexity of the preprocessing step involved. Fifth, we assess the practical impact of multiple combinations of the tightening strategies. The proposed tightening strategies enables the identification of an optimal \gls{cpwl} approximation in a relatively low computation time (\textit{i.e.}, up to 20 times faster than the existing literature).

The remainder of the article is organized as follows: Section 2 introduces the theoretical background of the \gls{cpwl}-fitting problem. Section 3 describes the proposed \gls{milp} formulation. Section 4 formalizes the concept of well-behaved \gls{cpwl} fitting and demonstrates the fact that any \gls{cpwl} interpolation or approximation of a data set has a well-behaved version. Section 5 introduces the six tightening strategies of the \gls{milp} formulation. Section 6 introduces the case studies and asses the performances of several combinations of tightening strategies. Section 7 presents the conclusions.

\section{Background}

Let $S = (\bm{x}_{i},z_i)_{i=1,...,N} = (x_{i,1}, ..., x_{i,d},z_i)_{i=1,...,N}$ be a set of $N$ points of $\mathbb{R}^{d+1}$. We assume that the points $(\bm{x}_{i})_{i=1,...,N}$ are in general position in $\mathbb{R}^{d}$, \textit{i.e.}, any subset of $d+1$ points is affinely independent (\cite{edelsbrunner_constructing_1986}). Let $P_{\mathbb{R}^{d}}: \mathbb{R}^{d+1} \rightarrow \mathbb{R}^{d}$ be the projection defined as $P_{\mathbb{R}^{d}}(x_{1}, ..., x_{d}, z) = (x_{1}, ..., x_{d})$. Let $S_{\mathbb{R}^{d}} = (\bm{x}_{i})_{i=1,...,N}$ be the projection of $S$ by $P_{\mathbb{R}^{d}}$. Let $D = Conv( S_{\mathbb{R}^{d}} )$ be the convex hull of $S_{\mathbb{R}^{d}}$.

\begin{definition}
\label{defCPWL}
We say that $f: D \rightarrow \mathbb{R}$ is a \gls{cpwl} function when $f$ is continuous and there is a set of $P$ affine functions $\{ f_{k}: D_{k} \rightarrow \mathbb{R}, k=1,...,P\}$ such that the $D_k$ are compact,  $\bigcup\limits_{k=1}^{P} D_{k} = D$, and $f_{|D_{k}} = f_{k}, \forall k \in \{1,...,P\}$. The $f_{k}$ are called the affine pieces of $f$, and the $D_{k}$ are called the affine domains of $f$.
\end{definition}

Note that, contrary to the more common definition of \gls{cpwl} functions from \cite{geisler_using_2012}, we do not assume that the $D_{k}$ form a partition of $D$. The intersection of two affine domains is not required to have an empty interior. In addition, an affine domain can be empty, and two affine domains can be equal as long as the corresponding affine pieces are equal. This implies that the set of affine functions representing the \gls{cpwl} function is not unique. However, apart from this, \cref{defCPWL} describes exactly the same mathematical object as the \gls{cpwl} functions defined in the literature. The flexibility of our definition enables us to compare \gls{cpwl} functions to one another, which will be useful in the following section.

\begin{fact}
\label{fact1}
Let $f: D \rightarrow \mathbb{R}$ be a \gls{cpwl} function. $f$ can be expressed as the difference of two convex \gls{cpwl} functions $f^+$ and $f^-$. Specifically, there is a set of $P^+$ affine functions $\{f_j^+: D \rightarrow \mathbb{R}, j \in \{1,...,P^+\}\}$ and $P^-$ affine functions $\{f_k^-: D \rightarrow \mathbb{R}, k \in \{1,...,P^-\}\}$ such that:
$$f(\bm{x}) = f^+(\bm{x}) - f^-(\bm{x}) = \max_{j \in \{1,...,P^+\}}^{} f_j^+(\bm{x}) - \max_{k \in \{1,...,P^-\}}^{} f_k^-(\bm{x}), \forall \bm{x} \in D$$

$f_j^c$ are called the affine pieces of $f^c$, and $D_j^c = \{\bm{x} \in D: f^c(\bm{x}) = f_j^c(\bm{x})\}$ are called the affine domains of $f^c, j \in \{1,...,P^c\}, c \in \{+,-\}$. In addition, $\{D_j^c: j \in \{1,...,P^c\}\}$ is a partition of $D$.
\end{fact}

This fact is demonstrated in \cite{kripfganz_piecewise_1987}. The above formulation is called the \gls{dc} representation of $f$. Note that the \gls{dc} representation of a \gls{cpwl} function is not unique.

\begin{fact}
\label{fact2}
Let $f: D \rightarrow \mathbb{R}$ be a \gls{cpwl} function. Let $f = f^+ - f^-$ be a DC representation of $f$. The affine domains of $f^+$ and $f^-$ are compact. In addition, any affine domain of $f$ is the intersection of an affine domain of $f^+$ and $f^-$. In other words, for any affine domain $D_p$ of $f$, there is an affine domain $D_j^+$ and $D_k^-$ of $f^+$ and $f^-$ such that $D_p = D_j^+ \cap D_k^-$, and $f(\bm{x}) = f_j^+(\bm{x}) - f_k^-(\bm{x}), \forall x \in D_p$.
\end{fact}

This fact is demonstrated in \cite{kripfganz_piecewise_1987} and stems from the DC representation of $f$.

\begin{definition}
We say that $f: D \rightarrow \mathbb{R}$ is an interpolation of, or interpolates, $S$ when $f(\bm{x}_{i}) = z_{i}, \forall i \in \{1,...,N\}$ (\cite{davis_interpolation_1975}).
\end{definition}

\begin{definition}
We say that $f: D \rightarrow \mathbb{R}$ is an $\varepsilon-$approximation of, or $\varepsilon-$approximates, $S$ when there exists $(e_{i})_{i=1,...,N} \in \mathbb{R}^N$ such that $|e_{i}| \leq \varepsilon$ and $f$ interpolates the set $(\bm{x}_{i}, z_{i} + e_{i})_{i=1,...,N}$. In other words, $f$ is an $\varepsilon-$approximation of $S$ if $\max_{i=1,...,N}{|f(\bm{x}_i) - z_i|} \leq \varepsilon$.
\end{definition}

\begin{definition}
\label{definition_equivalent}
Let $f, g: D \rightarrow \mathbb{R}$ and $g: D \rightarrow \mathbb{R}$ be two $\varepsilon-$approximations of $S$. We say that $f$ and $g$ are equivalent with respect to $S$ when $f(\bm{x}_{i}) = g(\bm{x}_{i}), \forall i \in \{1,...,N\}$.
\end{definition}

\section{MILP formulation}

%\subsection{MILP formulation of the optimal CPWL fitting problem}
We aim to identify the optimal \gls{cpwl} fitting of a set of points $S$. More specifically, we aim to identify the \gls{cpwl} $\varepsilon-$approximation of $S$ that minimizes an objective $Q$ subject to a given error tolerance $\varepsilon$. For example, $Q$ can be the maximum fitting error, the average fitting error, or the number of affine pieces.

The following \gls{milp} formulation is largely based on the formulation presented in \cite{kazda_nonconvex_2021}:

\begin{equation}
\label{objectiveQ}
    \min Q
\end{equation}
$s.t.$
\begin{equation}
\label{DCequation}
    f(\bm{x}_i) = f^+(\bm{x}_i) - f^-(\bm{x}_i), \quad i \in \{1,...,N\}
\end{equation}
\begin{equation}
\label{ConvexEquation0}
    0 \leq f^c(\bm{x}_i) - {\bm{a}_j^c}^T \bm{x}_i - b_j^c, \quad i \in \{1,...,N\}, \quad j \in \{1,...,P^c\}, \quad c \in \{+,-\}
\end{equation}
\begin{equation}
\label{ConvexEquation}
    f^c(\bm{x}_i) - {\bm{a}_j^c}^T \bm{x}_i - b_j^c \leq M_i^c (1-\delta_{i,j}^c), \quad i \in \{1,...,N\}, \quad j \in \{1,...,P^c\}, \quad c \in \{+,-\}
\end{equation}
\begin{equation}
\label{numberplaneperpoint}
    \sum\limits_{j=1}^{P^c} \delta_{i,j}^c \geq 1, \quad i \in \{1,...,N\}, \quad c \in \{+,-\}
\end{equation}
\begin{equation}
\label{FittingEquation}
    -e_i \leq f(\bm{x}_i) - z_i \leq e_i, \quad  i \in \{1,...,N\}
\end{equation}
\begin{equation}
\label{ErrorBound}
    0 \leq e_i \leq \varepsilon, \quad  i \in \{1,...,N\}
\end{equation}
\begin{equation}
\label{DeltaBinary}
    \delta_{i,j}^c \in \{0,1\}, \quad  i \in \{1,...,N\}, \quad j \in \{1,...,P^c\}, \quad c \in \{+,-\}
\end{equation}

\Cref{objectiveQ} represents the objective function to be minimized. As long as $Q$ is a linear expression, \cref{objectiveQ,DCequation,ConvexEquation0,ConvexEquation,numberplaneperpoint,FittingEquation,ErrorBound,DeltaBinary} formulate a \gls{milp} problem. For the rest of the paper, we will refer to this \gls{milp} problem as $MILP1(Q)$. \Cref{DCequation} is the DC representation of the \gls{cpwl} function, where $f(\bm{x}_i), f^+(\bm{x}_i), f^-(\bm{x}_i)$ are linear variables. \Cref{ConvexEquation0,ConvexEquation} formulate the representation of each convex function $f^c$ as the maximum of a set of affine functions $f_j^c$, where $\bm{a}_j^c \in \mathbb{R}^d$ and $b_j^c \in \mathbb{R}$ are the linear coefficients and bias terms of each affine piece $f_j^c$, and are linear variables in $MILP1(Q)$. $P^+$ and $P^-$ are the numbers of affine pieces of $f^+$ and $f^-$, respectively. \Cref{numberplaneperpoint} ensures that each point $\bm{x}_i$ belongs to at least one affine domain of $f^+$ and one affine domain of $f^-$. Note that, contrary to the literature (e.g. \cite{kazda_nonconvex_2021}), we use an inequality constraint instead of an equality constraint. By doing so, we allow the points $\bm{x}_i$ to belong to multiple affine domains. This can occur if, for example, some affine pieces are identical or if some points are at the border between two or more affine domains. This inequality is also critical to ensure that some well-behaved \gls{cpwl} solutions are feasible in $MILP1(Q)$, which is a property that will be discussed in the next section.  \Cref{FittingEquation} defines the approximation error at each point $\bm{x}_i$ as the distance between $z_i$ and the value of the \gls{cpwl} function in $\bm{x}_i$. \Cref{ErrorBound} ensures that this distance is never greater than the specified error tolerance $\varepsilon$. Finally, \cref{DeltaBinary} defines $\delta_{i,j}^c$ as a binary variable. This binary variable serves as an indicator variable which is equal to 1 when the point $\bm{x}_i$ belongs to the domain $D_j^c$. Note that the big-M parameter $M_i^c$ in \cref{ConvexEquation} must be large enough so it does not constrain the term $f^c(\bm{x}_i) - {\bm{a}_j^c}^T \bm{x}_i - b_j^c$ when $\delta_{i,j}^c = 0$. In \cref{tighteningsection}, we will identify a tight value for $M_i^c$. 

Note that some optimization solvers like Gurobi (\cite{gurobi_optimization_llc_gurobi_2024}) make it possible to formulate \cref{ConvexEquation} without using a big-M parameter . Instead, \cref{ConvexEquation} can be replaced by the indicator constraint below (\cite{gurobi_optimization_llc_modeladdgenconstrindicator_2024}):
\begin{equation}
\label{ConvexEquationIndicator}
    \delta_{i,j}^c = 1 \Rightarrow f^c(\bm{x}_i) - {\bm{a}_j^c}^T \bm{x}_i - b_j^c \leq 0, \quad  i \in \{1,...,N\}, \quad j \in \{1,...,P^c\}, \quad c \in \{+,-\}
\end{equation}
We note $MILPIC1(Q)$ the alternative \gls{milp} formulation using the indicator constraint, which is composed of \cref{objectiveQ,DCequation,ConvexEquation0,ConvexEquationIndicator,numberplaneperpoint,FittingEquation,ErrorBound,DeltaBinary}.

Adding \cref{eq_mean_error} to $MILP1(Q)$ formulates a \gls{milp} in which the average fitting error is minimized. Alternatively, adding \cref{eq_max_error} minimizes the maximum fitting error. 

\begin{equation}
\label{eq_mean_error}
    Q = \frac{1}{N}\sum_{i=1}^N e_i
\end{equation}
\begin{equation}
\label{eq_max_error}
    Q \geq e_i,\quad i \in \{1,...,N\}
\end{equation}

Other objective functions may include the number of affine pieces of  $f^+$, $f^-$, or $f$, as well as a combination of error metric and number of affine pieces. These alternative objective functions can be found in \cref{alternative_objectives}.

\section{Well-behaved CPWL interpolations}
This section introduces a new class of \gls{cpwl} interpolations called ``well-behaved'' \gls{cpwl} interpolations. Loosely speaking, well-behaved \gls{cpwl} interpolations are composed of affine pieces whose gradients are not excessively steep. This class of \gls{cpwl} interpolations aligns with the intuitive representation of what the \gls{cpwl} interpolation of a data set should look like, and they have desirable properties when modeled in $MILP1(Q)$. This section formalizes the definition of a well-behaved \gls{cpwl} interpolation and defines what constitutes a well-behaved version of a \gls{cpwl} interpolation. We then demonstrate that any \gls{cpwl} interpolation has a well-behaved version. This result is critical to the tightening procedure introduced in \cref{tighteningsection}.

\begin{fact}
\label{fact_linear_interpolation}
Let $S = (\bm{x}_{i},z_i)_{i=1,...,d+1}$ be a set of $(d+1)$ points of $\mathbb{R}^{d+1}$, where the $\bm{x}_{i}$ are in general position. There exists a unique linear interpolation of $S$. In other words, $S$ can be interpolated by a unique affine function. This is due to the fact that the system of linear equations 
$\bm{a}^T \bm{x}_{i} + b = z_i, i\in \{1,...,d+1\}$, with variables $\bm{a} \in \mathbb{R}^d$ and $b \in \mathbb{R}$, is exactly determined and has a unique solution due to the general position of the $\bm{x}_{i}$. If the number of points is less than $d+1$, the system is underdetermined and there is an infinite number of possible linear interpolations. If the number of points is greater than $d+1$, because the $\bm{x}_{i}$ are in general position, the system becomes overdetermined and there exists either one or no linear interpolation.
\end{fact}

\begin{definition}
Let $f: D \rightarrow \mathbb{R}$ be a \gls{cpwl} interpolation of $S$, with $|S|=N \geq d+1$. We say that an affine piece of $f$ is underdetermined, exactly determined, or overdetermined with respect to $S$ if it interpolates less than, exactly, or more than $d+1$ points of $S$, respectively.
\end{definition}

\begin{definition}
\label{def_well_behaved}
We say that $f: D \rightarrow \mathbb{R}$ is a well-behaved \gls{cpwl} interpolation of $S$ when $f$ is a \gls{cpwl} interpolation of $S$ and each affine piece of $f$ is exactly determined or overdetermined, i.e., each affine piece interpolates at least $d+1$ points of $S$. Additionally, let $f: D \rightarrow \mathbb{R}$ be an $\varepsilon$-approximation of $S$ and $e_i = f(\bm{x}_{i}) - z_i$. We say that $f$ is a well-behaved \gls{cpwl} $\varepsilon$-approximation of $S$ when $f$ is a well-behaved \gls{cpwl} interpolation of $(\bm{x}_{i},z_i+e_i)_{i=1,...,N}$.
\end{definition}
\begin{definition}
\label{def_well_behaved_version}
Let $f$ and $g$ be two \gls{cpwl} interpolations of $S$ composed of the set of affine pieces \{$f_{k}: D_{k} \rightarrow \mathbb{R}, k=1,...,P$\} and \{$g_{k}: D'_{k} \rightarrow \mathbb{R}, k=1,...,P$\}, respectively. We say that $g$ is a well-behaved version of the \gls{cpwl} interpolation $f$ with respect to $S$ if $g$ is a well-behaved \gls{cpwl} interpolation of $S$ and $D_{k} \cap S_{\mathbb{R}^{d}} \subseteq D'_{k} \cap S_{\mathbb{R}^{d}}, \forall k \in \{1,...,P\}$. Additionally, let $f$ and $g$ be two \gls{cpwl} $\varepsilon$-approximations of $S$ and $e_i = f(\bm{x}_{i}) - z_i$. We say that $g$ is a well-behaved version of the \gls{cpwl} $\varepsilon$-approximation $f$ with respect to $S$ if $g$ is a well-behaved version of the \gls{cpwl} interpolation $f$ with respect to $(\bm{x}_{i},z_i+e_i)_{i=1,...,N}$. 
\end{definition}

\begin{remark}
    If $g$ is a well-behaved version of the \gls{cpwl} $\varepsilon$-approximation $f$ with respect to $S$, then the $\varepsilon$-approximations $f$ and $g$ are equivalent according to \cref{definition_equivalent}.
\end{remark}

In other words, a well-behaved version of a \gls{cpwl} interpolation $f$ of $S$ is a transformation of $f$ where each affine piece has been adjusted, or ``tilted'', to interpolate at least $d+1$ points of $S$, including the points it was originally interpolating (before being tilted). An illustration for $d=1$ is provided in \cref{fig_wellbehaved}. In the following sections, we show that well-behaved \gls{cpwl} interpolations have interesting properties that can be used to tighten the \gls{milp} formulation. To the best of the authors' knowledge, this is the first time this class of \gls{cpwl} interpolations is introduced. 
We prove below that for any \gls{cpwl} interpolation, a well-behaved version always exists. 

\begin{figure}[h!]
    \centering
    \includegraphics[width=0.6\textwidth]{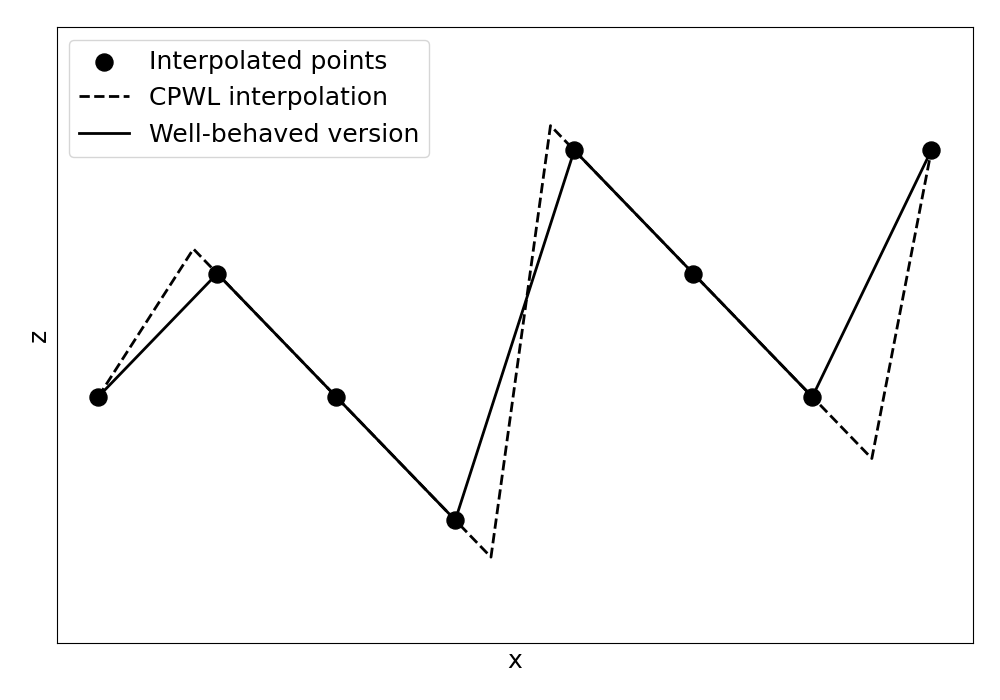}
    \caption{Illustration of a CPWL interpolation and its well-behaved version for $d=1$.}
    \label{fig_wellbehaved}
\end{figure}

% \begin{figure}
%  {\includegraphics[height=2.2in]{WellBehaved.png}}
% {Illustration of a CPWL interpolation and its well-behaved version for $d=1$. \label{fig_wellbehaved}}
% {Both CPWL functions interpolate the points and are composed of five affine pieces, but the affine pieces of the well-behaved interpolation have been tilted to interpolate at least two points per piece.}
% \end{figure}

\begin{remark}
    If $f$ is well-behaved, then $f$ is a well-behaved version of itself.
\end{remark}

\begin{lemma}
\label{lemma_neighbor_inequality}
    Let $S$ be a set of $N$ points in $\mathbb{R}^{d+1}$, with $N \geq d+1$. Let $f: D \rightarrow \mathbb{R}$ be a \gls{cpwl} interpolation of $S$. Let $f_1: D_1 \rightarrow \mathbb{R}$ be one of the affine pieces of $f$, and let $f_2: D_2 \rightarrow \mathbb{R}$ be a neighboring affine piece of $f_1$, i.e., an affine piece of $f$ sharing a common domain boundary with $f_1$. Let $\widetilde{f_1}$ be the affine extension of $f_1$ to the domain $D$. Then,
    $$\exists c \in \{-1,1\}: \quad \forall \bm{x} \in D_2, \quad c \left( \widetilde{f_1}(\bm{x}) - f_2(\bm{x}) \right) \leq 0$$
\end{lemma}

\begin{proof}[Proof of \cref{lemma_neighbor_inequality}] Let $\bm{x} \in D_2$. According to \cref{fact2}, a shared domain boundary between $f_1$ and $f_2$ is either a shared domain boundary between two affine pieces of $f^+$ or between two pieces of $f^-$. \underline{Case 1}: The shared domain boundary is between two affine pieces of $f^+$. Then, using \cref{fact2}, we can write $\widetilde{f_1}(\bm{x})=f_1^+(\bm{x})-f_0^-(\bm{x})$ and $f_2(\bm{x})=f_2^+(\bm{x})-f_0^-(\bm{x})$. Since $f_2^+(\bm{x}) = \max_{j \in \{1,...,P^+\}}f_j^+(\bm{x}) \geq f_1^+(\bm{x})$ we therefore have $\widetilde{f_1}(\bm{x}) \leq f_2(\bm{x})$. Thus $\left( \widetilde{f_1}(\bm{x}) - f_2(\bm{x}) \right) \leq 0$. \underline{Case 2}: The shared domain boundary is between two affine pieces of $f^-$. We can write $\widetilde{f_1}(\bm{x})=f_0^+(\bm{x})-f_1^-(\bm{x})$ and $f_2(\bm{x})=f_0^+(\bm{x})-f_2^-(\bm{x})$. Since $f_2^-(\bm{x}) \geq f_1^-(\bm{x})$, we therefore have $\widetilde{f_1}(\bm{x}) \geq f_2(\bm{x})$. Thus, $(-1)\left( \widetilde{f_1}(\bm{x}) - f_2(\bm{x}) \right) \leq 0$.
\end{proof}

\begin{fact}
\label{active_constraints_polyhedron}
    Let $\bm{A} \in \mathbb{R}^{m \times n}$ and $\bm{b} \in  \mathbb{R}^m$. Assume that the $m$ rows of $\bm{A}$ are linearly independent and the polyhedron $U = \{ \bm{x} \in \mathbb{R}^n : \bm{A} \bm{x} \leq \bm{b}\}$ is non-empty. Then, there exists a point $\bm{x}^* \in \mathbb{R}^n$ such that $\bm{A}_i \bm{x}^* = b_i$ for $\min(n,m)$ rows $\bm{A}_i$ of $\bm{A}$. 
    
    The $\min(n,m)$ constraints ``$\bm{A}_i \bm{x}^* = b_i$'' are also known as the active constraints of the solution $\bm{x}^*$. The point $\bm{x}^*$ lies on the boundary of $U$ and, if $m \geq n$, $\bm{x}^*$ is a vertex of $U$ and a basic feasible solution of $U$. This is a well-established fact from LP theory, as described in \cite{bertsimas_introduction_1997}.
\end{fact}

\begin{lemma}
\label{lemma_add_points}
 Let $S$ be a set of $N$ points in $\mathbb{R}^{d+1}$, with $N \geq d+1$. Let $f: D \rightarrow \mathbb{R}$ be a \gls{cpwl} interpolation of $S$ that is not well-behaved, i.e., at least one of the affine pieces of $f$ is underdetermined. Let $f_0: D_0 \rightarrow \mathbb{R}$ represent an underdetermined affine piece of $f$ that interpolates $n$ points of $S$ ($n \leq d$). Let $f_j: D_j \rightarrow \mathbb{R}, j \in \{1,...,J\}$ be the $J$ neighboring pieces of $f_0$, i.e., the affine pieces that share a domain boundary with $f_0$. Let $\widetilde{D_0} = \bigcup_{j=0}^JD_j$. Assume there is a total of $m$ points in $S$ whose projections lie in $\widetilde{D_0} \setminus D_0$ (interpolated by the J neighboring pieces but not by $f_0$). Then, there exists an alternative \gls{cpwl} interpolation of $S$ in which $f_0$ is adjusted to interpolate $\min(d+1,n+m)$ points of $S$ without impacting the points interpolated by the neighboring pieces. In other words, we can construct a new \gls{cpwl} interpolation of $S$ defined over $\widetilde{D_0}$ and represented by a set of affine functions $\{f'_j: D'_j \rightarrow \mathbb{R}, j \in \{0,1,...,J\}\}$, such that: (1) The functions ${f'_j}$ and $f_j$ are equal on their shared domain, i.e., ${f'_j}(\bm{x}) = {f_j}(\bm{x}), \forall \bm{x} \in D'_j \cap D_j, j \geq 1$, (2) The points that are interpolated by $f_j$ and ${f'_j}$ are the same , i.e., $S_{\mathbb{R}^{d}} \cap D'_j = S_{\mathbb{R}^{d}} \cap D_j, j \geq 1$, and (3) The function $f'_0$ interpolates $\min(d+1,n+m)$ points of $S$, which include the $n$ points interpolated by $f_0$, i.e., $|S_{\mathbb{R}^{d}} \cap D'_0| = \min(d+1,n+m)$ and $S_{\mathbb{R}^{d}} \cap D_0 \subseteq S_{\mathbb{R}^{d}} \cap D'_0$.
\end{lemma}

\begin{proof}[Proof of \cref{lemma_add_points}]
    Let $\bm{a}_0$ and $b_0$ be the linear coefficients and bias terms of $f_0$, \textit{i.e.}, $f_0(\bm{x}) = \bm{a_0}^T \bm{x} + b_0, \forall \bm{x} \in D_0$. Let $(\bm{x}_i,z_i)_{i=1,...,n}$ denote the $n$ points of $S$ whose projections on $\mathbb{R}^d$ lie in $D_0$, and $(\bm{x}'_i,z'_i)_{i=1,...,m}$ denote the $m$ points of $S$ whose projections on $\mathbb{R}^d$ lie in $\widetilde{D_0} \setminus D_0$. Let $p(i)$ denote the smallest integer such that $\bm{x}'_i \in D_{p(i)}$.
    According to \cref{lemma_neighbor_inequality}, $ \exists (c_j)_{j=1,...,J} \in \{-1,1\}^J: c_{p(i)}\left( \widetilde{f_0}(\bm{x}'_i) - f_{p(i)}(\bm{x}'_i) \right) \leq 0, \forall i \in \{1,...,m\}$. Since $\widetilde{f_0}(\bm{x}'_i) = \bm{a_0}^T \bm{x}'_i + b_0$ and $f_{p(i)}(\bm{x}'_i) = z'_i$, this can be expressed as: $c_{p(i)}\left( \bm{a_0}^T \bm{x}'_i + b_0 - z'_i \right) \leq 0, \forall i \in \{1,...,m\}$. In other words, $(\bm{a}_0,b_0)$ is a point of the polyhedron $U$ defined by:
    $$(\bm{a},b) \in \mathbb{R}^{d+1}:$$
    $$\bm{a}^T \bm{x}_i + b - z_i = 0, \quad i \in \{1,...,n\}$$
    $$c_{p(i)}\left( \bm{a}^T \bm{x}'_i + b - z'_i \right) \leq 0, \quad i \in \{1,...,m\}$$
    $U$ is not empty because $(\bm{a}_0,b_0) \in U$. Therefore, according to \cref{active_constraints_polyhedron}, there exists a solution $(\bm{a}^*,b^*) \in U$ with $\min(d+1,n+m)$ active constraints. Note that the constraints ``$\bm{a}^{*T} \bm{x}_i + b^* - z_i = 0$'' are already active, which means that there must be $\min(d+1,n+m)-n$ additional active constraints of type ``$c_{p(i)}\left( \bm{a}^{*T} \bm{x}'_i + b^* - z'_i \right) = 0$''. We define $f'_0$ as the affine function defined by $f'_0(\bm{x}) = \bm{a}^{*T} \bm{x} + b^*$. We define the domain of $f'_0$ as the domain delimited by the $J$ boundaries $\{\bm{x} \in D: f'_0(\bm{x}) = \widetilde{f_j}(\bm{x})\}$. We define $f'_j$ as the $J$ affine pieces $f_j$ for which the domains were updated according to the $J$ boundaries of $f'_0$. The set of affine functions $\{f'_j: D'_j \rightarrow \mathbb{R}, j \in \{0,1,...,J\}\}$ constitutes a valid \gls{cpwl} interpolation of $S$ over $\widetilde{D_0}$, with $f'_0$ interpolating $\min(d+1,n+m)$ points of $S$, including the original $n$ points interpolated by $f_0$. The points of $S$ interpolated by $f'_j$ are the same as the points interpolated by $f_j$, $\forall j \in \{1,...,J\}$.
\end{proof}

\begin{theorem}
    \label{theorem_well_behaved}
    Let $S$ be a set of $N$ points in $\mathbb{R}^{d+1}$, where $N \geq d+1$. Any \gls{cpwl} interpolation of $S$ has a well-behaved version.
\end{theorem}

\begin{proof}[Proof of \cref{theorem_well_behaved}]
    Let $f: D \rightarrow \mathbb{R}$ be a \gls{cpwl} interpolation of $S$. If each affine piece of $f$ interpolates at least $d+1$ points of $S$, then $f$ is a well-behaved version of itself. Otherwise, $f$ has at least one underdetermined affine piece, i.e., interpolating less than $d+1$ points. Let $f_0$ be an underdetermined affine piece of $f$ interpolating $n$ points of $S$ such that one of the neighboring pieces of $f_0$ interpolates at least one point of $S$ not interpolated by $f_0$. $f_0$ exists because $n \leq d$ and $d+1 \leq N$. According to \cref{lemma_add_points}, we can build an alternative \gls{cpwl} interpolation of $S$ in which $f_0$ has been adjusted to interpolate at least $n+1$ points of $S$ without affecting the points interpolated by its neighboring affine pieces. The process of finding and adjusting an underdetermined affine piece can be repeated until no underdetermined piece remains. Since the number of points in $S$ is finite, this procedure terminates after a finite number of steps. At the end of the process, every affine piece interpolates at least $d+1$ points of $S$, and the final \gls{cpwl} interpolation is a well-behaved version of $f$. 
\end{proof}

\begin{corollary}
\label{corollary_1}
    Any \gls{cpwl} $\varepsilon$-approximation of $S$ has a well-behaved version.
\end{corollary}

\begin{proof}[Proof of \cref{corollary_1}]
    The proof directly follows from \cref{theorem_well_behaved} and \cref{def_well_behaved_version}.
\end{proof}

\begin{remark}
    Note that the procedure of tilting the affine pieces of $f$ to construct a well-behaved version may results in two or more underdetermined pieces merging together. In other words, initially distinct pieces may end up interpolating the same subset of points and having the same expression and affine domain after converting $f$ to a well-behaved interpolation. In this case, the well-behaved version is still a valid \gls{cpwl} function owing to the flexibility of \cref{defCPWL}.
\end{remark}

\begin{remark}
    A well-behaved version of a \gls{cpwl} interpolation is not necessarily unique. That is, a \gls{cpwl} interpolation may have several well-behaved versions. This stems from the fact that the point with $\min(n,m)$ active constraints in \cref{active_constraints_polyhedron} is not necessarily unique.
\end{remark}

An important application of this theorem is that we can significantly reduce the set of \gls{cpwl} functions to be considered in $MILP1(Q)$ without impacting the feasibility of the problem. Specifically, if a \gls{cpwl} $\varepsilon$-approximation $f$ of $S$ exists, a well-behaved version $g$ of $f$ also exists. The two \gls{cpwl} solutions $f$ and $g$ have the same number of affine pieces and are equivalent with respect to $S$, i.e., $f(\bm{x}_{i}) = g(\bm{x}_{i}), \forall i \in \{1,...,N\}$. Therefore, we can eliminate \gls{cpwl} solutions that are not well-behaved from the feasible region of $MILP1(Q)$ since they are redundant solutions and may have excessively steep gradients due to their underdetermined affine pieces. In the next section, we will see how this theorem allows us to derive valid tight bounds and constraints for the \gls{milp} problem.

\begin{definition}
    Herein, we will refer to the set of \gls{cpwl} $\varepsilon$-approximations of $S$ as $CPWL(S,\varepsilon)$. The subset of $CPWL(S,\varepsilon)$ that can be represented using $P^+, P^-,$ affine pieces for $f^+, f^-$, respectively, is denoted $CPWL(S,\varepsilon,P^+,P^-)$. The set of well-behaved \gls{cpwl} $\varepsilon$-approximations of $S$ will be denoted as $CPWL^*(S,\varepsilon)$, and $CPWL^*(S,\varepsilon,P^+,P^-)$ when represented by $P^+, P^-$ affine pieces. Note that $CPWL^*(S,\varepsilon) \subset CPWL(S,\varepsilon)$. Additionally, note that the feasible region of the \gls{milp} problem $MILP1(Q)$ is given by $CPWL(S,\varepsilon,P^+,P^-)$.
\end{definition}

\section{Tightening the MILP formulation}
\label{tighteningsection}

This section introduces six strategies to tighten and enhance the formulation of $MILP1(Q)$. First, we demonstrate that we can fix one of the affine pieces of $f^-$ without affecting the set of \gls{cpwl} solutions. Next, we impose an ordering on the affine pieces of $f^+$ and $f^-$ that does not affect the set of \gls{cpwl} solutions. We can also impose that each affine piece of $f^+$ and $f^-$ contains at least $d+1$ points, which serves as a valid tightening of the set of feasible well-behaved \gls{cpwl} solutions. Alternatively, we may require that each affine piece of $f$ contains at least $d+1$ points, which requires the use of additional variables.  Finally, we identify tight values for the big-M parameters and establish tight bounds for the linear variables.

Additional strategies that leverage the convexity of the functions $f^+$ and $f^-$ and of their affine domains are included in \cref{section_other_strategies}. However, contrary to the six strategies presented here, their impact on the feasible region is unclear and they do not seem to have a significant impact on the solution time.

Herein, it is assumed that $N \geq d+1$.

\subsection{Fixing one affine piece}

\begin{fact}
\label{fact4}
    Let $\{f_k, k \in \{0,1,...,K\}\}$ be a set of functions. We have $\max_{k \in \{1,...,K\}}(f_k-f_0) = \max_{k \in \{1,...,K\}}(f_k) - f_0$. 
    This fact follows directly from the translation invariance of the max function, i.e., $\max(a-c,b-c) = \max(a,b) - c, \forall a,b,c \in \mathbb{R}$. 
\end{fact}

\begin{theorem}
\label{theorem_first_plane_zero}
    The following equation constitutes a valid tightening constraint of $MILP1(Q)$:
    \begin{equation}
    \label{first_plane_zero}
        \bm{a}_1^-= \bm{0}, \quad b_1^-=0
    \end{equation}
\end{theorem}

\begin{proof}[Proof of \cref{theorem_first_plane_zero}]
    To prove this, we must demonstrate that any \gls{cpwl} function $f$ can be expressed as the difference of two convex \gls{cpwl} functions $f = \max_{j \in \{1,...,P^+\}}^{} g_j^+ - \max_{j \in \{1,...,P^-\}}^{} g_j^-$, with $g_1^-=0$. According to \cref{fact1}, there is a set of affine functions $\{f_1^+, ..., f_{P^+}^+,f_1^-, ..., f_{P^-}^-\}$ such that $f$ can be expressed as $f = \max_{j \in \{1,...,P^+\}}^{} f_j^+ - \max_{k \in \{1,...,P^-\}}^{} f_k^-$. Let $g_j^c = f_j^c - f_1^-, \forall j \in \{1,...,P^c\}, c \in \{+,-\}$. The function $g_j^c$ is affine since it is the difference of two affine functions.  Moreover, $g_1^-=f_1^--f_1^-=0$. According to \cref{fact4}, we have $\max_{j \in \{1,...,P^c\}}^{} (g_j^c) = \max_{j \in \{1,...,P^c\}}^{} (f_j^c) - f_1^-$. Therefore, we can rewrite $f$ as $f = \max_{j \in \{1,...,P^+\}}^{}(f_j^+) - \max_{k \in \{1,...,P^-\}}^{}(f_k^-) + f_1^- - f_1^- = \max_{j \in \{1,...,P^+\}}^{}(g_j^+) - \max_{k \in \{1,...,P^-\}}^{}(g_k^-)$. 
\end{proof}

Note that \cref{first_plane_zero} affects not only the set of possible values for $\bm{a}_1^-, b_1^-$ but also for all $\bm{a}_j^c, b_j^c$. Also note that the fixed values of 0 for $\bm{a}_1^-, b_1^-$ are chosen arbitrarily; any value could be imposed and serve as a valid tightening constraint for the problem. Imposing \cref{first_plane_zero} effectively reduces the dimensionality of the search space by $(d+1)$. Formally, \cref{first_plane_zero} does not introduce any additional linear constraints, instead, it introduces tight bounds on $(d+1)$ linear variables.

\subsection{Sorting the affine pieces}

Let $\bm{a}_j^c = \left (
\begin{array}{c}
a_{j,1}^c \\
\vdots \\
a_{j,d}^c
\end{array}
\right ) \in \mathbb{R}^d$ denote the linear coefficients of the affine piece $f_j^c$.

\begin{theorem}
\label{theorem_plane_ordering}
    The following equation constitutes a valid tightening constraint of $MILP1(Q)$:
    \begin{equation}
    \label{plane_ordering}
        a_{j,1}^c \leq a_{j+1,1}^c, \quad j \in \{1,...,P^c-1\}, \quad c \in \{+,-\}
    \end{equation}
\end{theorem}

\begin{proof}[Proof of \cref{theorem_plane_ordering}]
    The problem $MILP1(Q)$ exhibits symmetries with respects to the group of variables $(\bm{a}_j^+,b_j^+)_{j=1,...,P^+}$ and $(\bm{a}_k^-,b_k^-)_{j=1,...,P^-}$, respectively. This means that permuting the affine pieces of $f^+$, 
    or $f^-$, does not affect the solution of the problem. Therefore, we can arbitrarily sort the affine pieces of $f^+$ and $f^-$ based on the ascending order of the values $(a_{j,1}^+)_{j=1,..,P^+}$ and $(a_{k,1}^-)_{k=1,..,P^-}$, respectively.
\end{proof}

It is important to note that, after sorting the affine pieces of $f^+$ and $f^-$, the procedure outlined in the proof of \cref{theorem_first_plane_zero} can be applied to impose \cref{first_plane_zero} with no impact on the \gls{cpwl} solution. Thus, \cref{first_plane_zero,plane_ordering} are compatible and can be applied simultaneously with no impact on the \gls{cpwl} solution. The number of possible arrangements for $P^c$ affine pieces is $P^c!$, which implies that implementing \cref{plane_ordering} reduces the feasible region to a by a factor $P^+!P^-!$ of its original extent. Additionally, applying \cref{plane_ordering} introduces $P^+ + P^- - 2$ constraints to the \gls{milp} problem.

\subsection{Imposing $d+1$ points per affine domain of $f^c$}

\begin{theorem}
\label{theorem_number_points_per_plane}
    Adding the following equation to $MILP1(Q)$ reduces the feasible region to a superset of $CPWL^*(S,\varepsilon,P^+,P^-)$. In other words, some non well-behaved solutions may be eliminated from the new feasible region but all well-behaved solutions are preserved.
    \begin{equation}
    \label{number_points_per_plane}
        \sum\limits_{i=1}^{N} \delta_{i,j}^c \geq d+1, \quad j \in \{1,...,P^c\},  \quad c \in \{+,-\}
    \end{equation}
\end{theorem}

\begin{proof}[Proof of \cref{theorem_number_points_per_plane}]
    Let $f = f^+ - f^-$ be a well-behaved \gls{cpwl} solution. Let $D_j^+$ be one of the affine domains of $f^+$. According to \cref{fact1}, $Int(D_j^+) \neq \varnothing$ and $D = \bigcup_{k=1}^{P^-}D_k^-$. Therefore, there exists a $k$ such that $Int(D_j^+ \cap D_k^-) \neq \varnothing$. According to \cref{fact2}, $D_p = D_j^+ \cap D_k^-$ is an affine domain of $f$. According to \cref{def_well_behaved}, $|D_p \cap S_{\mathbb{R}^d}| \geq d+1$. Furthermore, $D_p \subset D_j^+$ implies $D_p \cap S_{\mathbb{R}^d} \subset D_j^+ \cap S_{\mathbb{R}^d}$. Therefore, we have $|D_j^+ \cap S_{\mathbb{R}^d}| \geq |D_p \cap S_{\mathbb{R}^d}| \geq d+1$. Similarly, for an affine domain $D_k^-$ of $f^-$ we obtain $|D_k^- \cap S_{\mathbb{R}^d}| \geq d+1$. In other words, each affine domain of $f^+$ and $f^-$ contains at least $d+1$ points. 
\end{proof}

Note that the inequality from \cref{numberplaneperpoint}, which allows each point to belong to more than one affine domain of $f^+$ or $f^-$, is critical for utilizing \cref{number_points_per_plane}. If we were instead imposing a single affine domain per point, as in the formulation from \cite{kazda_nonconvex_2021}, \cref{number_points_per_plane} would become too restrictive. In fact, the combined requirement of having only one affine domain per point and at least $(d+1)$ points per affine domain would imply that $N \geq \max(P^+,P^-)(d+1)$, making the \gls{milp} problem infeasible if that condition was not satisfied. Using \cref{number_points_per_plane} eliminates some non well-behaved \gls{cpwl} solutions from the feasible region. However, all well-behaved solutions are preserved. Owing to \cref{theorem_well_behaved}, it is proven that this does not impact the quality of the feasible \gls{cpwl} solutions, as any \gls{cpwl} $\varepsilon$-approximation of S has a well-behaved version. Implementing \cref{number_points_per_plane} adds $P^+ + P^-$ constraints to the \gls{milp} problem.

\subsection{Imposing $d+1$ points per affine domain of $f$}

Alternatively, we can impose a set of constraints that further reduces the feasible region to the set of well-behaved \gls{cpwl} solutions. This approach requires the introduction of two additional sets of variables $\beta_{i,j,k}$ and $\gamma_{j,k}$. In the following equations, the indices $i \in \{1,...,N\}$, $j \in \{1,...,P^+\}$, $k \in \{1,...,P^-\}$, unless otherwise specified.

\begin{theorem}
\label{theorem_nber_points_affine_plane}
    Adding the following equations to $MILP1(Q)$ reduces the feasible region to $CPWL^*(S,\varepsilon,P^+,P^-)$:
    \begin{equation}
    \label{beta_delta_plus}
        \beta_{i,j,k} \leq \delta_{i,j}^+
    \end{equation}
    \begin{equation}
    \label{beta_delta_minus}
        \beta_{i,j,k} \leq \delta_{i,k}^-
    \end{equation}
    \begin{equation}
    \label{beta_delta_plus_minus}
        \beta_{i,j,k} \geq \delta_{i,j}^+ + \delta_{i,k}^- - 1
    \end{equation}
    \begin{equation}
    \label{gamma_beta}
        \beta_{i,j,k} \leq \gamma_{j,k}
    \end{equation}
    % \begin{equation}
    % \label{gamma_sum_beta}
    %     \gamma_{j,k} \leq \sum_{i=1}^N\beta_{i,j,k}
    % \end{equation}
    \begin{equation}
    \label{sum_beta_gamma}
        \sum_{i=1}^N\beta_{i,j,k} \geq (d+1)\gamma_{j,k}
    \end{equation}
    \begin{equation}
    \label{beta_bounds}
        0 \leq \beta_{i,j,k} \leq 1
    \end{equation}
    \begin{equation}
    \label{gamma_bounds}
        0 \leq \gamma_{j,k} \leq 1
    \end{equation}
\end{theorem}

\begin{proof}[Proof of \cref{theorem_nber_points_affine_plane}]
\Cref{beta_delta_plus,beta_delta_minus,beta_delta_plus_minus,beta_bounds} define $\beta_{i,j,k}$ as an indicator variable which is equal to 1 if $\bm{x}_i \in D_j^+ \cap D_k^-$ and 0 otherwise, i.e., $\beta_{i,j,k} = \delta_{i,j}^+ \wedge \delta_{i,k}^-$. \Cref{gamma_beta,gamma_bounds} define $\gamma_{j,k}$ as an indicator variable which is equal to 1 if $D_j^+ \cap D_k^-$ contains at least one point, i.e., $\bigvee_{i=1}^N \beta_{i,j,k} = 1 \Rightarrow \gamma_{j,k} = 1$. Finally, \cref{sum_beta_gamma} formulates the condition that if one point is in $D_j^+ \cap D_k^-$ then $D_j^+ \cap D_k^-$ must contain at least $d+1$ points, i.e., $\gamma_{j,k}=1 \Rightarrow \sum_{i=1}^N\beta_{i,j,k} \geq (d+1)$. Formally, $\beta_{i,j,k}$ and $\gamma_{j,k}$ should be defined as binary variables. However, the tight constraints from \cref{beta_delta_plus,beta_delta_minus,beta_delta_plus_minus,gamma_beta,sum_beta_gamma} involving the binary variables $\delta_{i,j}^c$ ensure that $\beta_{i,j,k}$ and $\gamma_{j,k}$ can only take values in $\{0,1\}$. 
\end{proof}

\Cref{beta_delta_plus,beta_delta_minus,beta_delta_plus_minus,gamma_beta,sum_beta_gamma,beta_bounds,gamma_bounds} involve the use of $P^+P^-(N+1)$ additional variables and $P^+P^-(4N+1)$ additional constraints. These constraints restrict the feasible space to be $CPWL^*(S,\varepsilon,P^+,P^-)$. Note that the indicator variable $\gamma_{j,k}$ is necessary in \cref{sum_beta_gamma} because the intersection of an affine domain of $f^+$ and $f^-$ may be empty. 

\begin{remark}
    \Cref{number_points_per_plane} is implied by \cref{beta_delta_plus,beta_delta_minus,beta_delta_plus_minus,gamma_beta,sum_beta_gamma,beta_bounds,gamma_bounds}, meaning that the latter dominate the former. In other words, imposing $d+1$ points per affine domain of $f$ implies, and is more restrictive than, imposing $d+1$ points per affine domain of $f^c$.
\end{remark}

\subsection{Tightening the big-M parameters}
\label{section_bigM}
In order to identify tight big-M values, we first need to define several sets. Let $A$ denote the set of all affine functions $D \rightarrow \mathbb{R}$. Let $[S]^{d+1}$ denote the set of all subsets of $S$ composed of $d+1$ points. Let $s=(\bm{x}_{i_k},z_{i_k})_{k=1,...,d+1} \in [S]^{d+1}$. We define the set $B \left( s,\varepsilon \right) = \{(\bm{x}_{i_k},z_{i_k}+e_k)_{k=1,...,d+1}: (e_k)_{k=1,...,d+1} \in \{-\varepsilon,+\varepsilon\}^{d+1}\}$. We define the set $A_{\varepsilon,s} = \{g \in A: g$ $ \varepsilon-$approximates $s\}$ and $A^*_{\varepsilon,s} = \{g \in A: g $ interpolates $s' \in B(s,\varepsilon)\}$. We also note $A_\varepsilon(S) = \bigcup_{s \in [S]^{d+1}} A_{\varepsilon,s}$ and $A^*_\varepsilon(S) = \bigcup_{s \in [S]^{d+1}} A^*_{\varepsilon,s}$. For convenience, let $\neg{c} =      
    \begin{cases}
      ``-"&\text{if $c=``+"$}\\
      ``+"&\text{if $c=``-"$}\\
    \end{cases}$. Finally, given a \gls{cpwl} solution $f = f^+ - f^-$, we define $f_{j,k} = f_j^+ - f_k^-$.

\begin{proposition}
\label{number_affine_functions}
    The set $A^{*}_{\varepsilon}(S)$ contains at most $\binom{N}{d+1}2^{d+1}$ affine functions.
\end{proposition}
\begin{proof}[Proof of \cref{number_affine_functions}]
    First, we have $|[S]^{d+1}| = \binom{N}{d+1}$. Additionally, $|A^*_{\varepsilon,s}| = |B(s,\varepsilon)| = |\{-\varepsilon,+\varepsilon\}^{d+1}| = 2^{d+1}$.
    Therefore, $|A^*_{\varepsilon}(S)|
    = |\bigcup_{s \in [S]^{d+1}} A^*_{\varepsilon,s}|
    \leq \sum_{s \in [S]^{d+1}} |A^*_{\varepsilon,s}|
    = \binom{N}{d+1}2^{d+1}$.
\end{proof}

\begin{lemma}
\label{lemma4}
    Let $f \in CPWL^*(S,\varepsilon,P^+,P^-)$, $\bm{x} \in D$. Let $(j,k) \in \{1,...,P^+\} \times \{1,...,P^-\} :Int(D_{j}^+ \cap D_{k}^-) \neq \varnothing$. Then, we have $$\min_{g \in A^*_\varepsilon(S)}g(\bm{x}) \leq f_{j,k}(\bm{x}) \leq \max_{g \in A^*_\varepsilon(S)}g(\bm{x})$$
\end{lemma}

\begin{proof}[Proof of \cref{lemma4}]
    Since $f$ is well-behaved, $f_{j,k}$ is an affine piece of $f$ that $\varepsilon$-approximates a subset $s\in [S]^{d+1}$. Therefore, $f_{j,k} \in A_{\varepsilon,s} \subset A_\varepsilon(S)$. As a result, $\min_{g \in A_\varepsilon(S)}g(\bm{x}) \leq f_{j,k}(\bm{x}) \leq \max_{g \in A_\varepsilon(S)}g(\bm{x})$. What remains to be proven is that the extrema of $g(\bm{x})$ on $A_{\varepsilon}(S)$ can be found on $A^*_{\varepsilon}(S)$. Note that $\max_{g \in A_{\varepsilon,s}}g(\bm{x})$ is equal to the optimal objective value of the LP optimization problem $LP_{max}(s,\bm{x},\varepsilon)$:
    $$\max g(\bm{x}),
    \quad s.t. \quad g \in A,
    \quad -\varepsilon \leq g(\bm{x}_{i_k}) - z_{i_k} \leq \varepsilon, \quad (\bm{x}_{i_k},z_{i_k})\in s$$
    Similarly, $\min_{g \in A_{\varepsilon,s}}g(\bm{x})$ is equal to the optimal objective value of $LP_{min}(s,\bm{x},\varepsilon)$. In addition, $A_{\varepsilon,s}$ corresponds to the feasible region of $LP_{min}(s,\bm{x},\varepsilon)$ and $LP_{max}(s,\bm{x},\varepsilon)$, whereas $A^*_{\varepsilon,s}$ corresponds to the vertices, i.e. extreme points, of $A_{\varepsilon,s}$, where the optimal solution is located.
    Consequently, $\max_{g \in A_{\varepsilon}(S)}g(\bm{x}) = \max_{s \in [S]^{d+1}}(\max_{g \in A_{\varepsilon,s}}(g(\bm{x}))) = \max_{s \in [S]^{d+1}}(\max_{g \in A^*_{\varepsilon,s}}(g(\bm{x}))) = \max_{g \in A^*_{\varepsilon}(S)}g(\bm{x})$. 
    Similarly, $\min_{g \in A_{\varepsilon}(S)}g(\bm{x}) = \min_{g \in A^*_{\varepsilon}(S)}g(\bm{x})$.
\end{proof}

\begin{lemma}
\label{lemma5}
    Let $f \in CPWL^*(S,\varepsilon,P^+,P^-)$, $\bm{x} \in D$, $c \in \{+,-\}$, $(j,k) \in \{1,...,P^c\}^2$. Then,
    $$|f_j^c(\bm{x}) - f_k^c(\bm{x})| \leq 
    \min \left( P^c-1,P^{\neg{c}} \right) \left( \max_{g \in A^*_\varepsilon(S)}g(\bm{x}) - \min_{g \in A^*_\varepsilon(S)}g(\bm{x}) \right)$$
\end{lemma}

\begin{proof}[Proof of \cref{lemma5}]
We prove the case $c = ``+"$, the case $c = ``-"$ being symmetric. Let $\bm{x}_j \in D_j^+$, $\bm{x}_k \in D_k^+$. By convexity of $D$, there exists a line segment $L$ connecting $\bm{x}_j$ to $\bm{x}_k$ that is contained in $D$. $L$ crosses $T$ affine domains of $f$, each of which is the intersection of an affine domain of $f^+$ and $f^-$. Let $(j_m,k_m)_{m=1,...,T}$ be such that $L$ crosses the non-empty domains $D_{j_m}^+ \cap D_{k_m}^-$ in ascending order of $m$ when traversing from $\bm{x}_j$ to $\bm{x}_k$. An example is depicted in \cref{fig_path_domains} for $d=2$. We have $j_1 = j$ and $j_T=k$. Without loss of generality, we assume that at each step $m$, the line segment $L$ transitions to either a new affine domain of $f^+$ or a new affine domain of $f^-$, but not both. In other words, $(j_m \neq j_{m+1}) \oplus (k_m \neq k_{m+1}), \forall m=1,..., T-1$. Since the intersection of $L$ with an affine domain is convex, $\{m : j_m = j\}$ and $\{m : k_m = k\}$ are sets of consecutive numbers, i.e., without gaps. Consequently, $j_m$ and $k_m$, can change value at most $P^+-1$ and $P^--1$ times, respectively. That is, $|\{m: j_m \neq j_{m+1}\}| \leq P^+-1$. Thus, we have:
    \begin{gather*}
    f_j^+(\bm{x}) - f_k^+(\bm{x}) 
    = f_{j_1}^+(\bm{x}) - f_{j_T}^+(\bm{x})
    \underset{(a)}{=} \sum_{m=1}^{T-1}(f_{j_{m}}^+(\bm{x}) - f_{j_{m+1}}^+(\bm{x}))
    \underset{(b)}{=} \sum_{m: j_m \neq j_{m+1}}(f_{j_{m}}^+(\bm{x}) - f_{j_{m+1}}^+(\bm{x}))\\
    \underset{(c)}{=} \sum_{m: j_m \neq j_{m+1}}(f_{j_{m}}^+(\bm{x}) - f_{k_{m}}^-(\bm{x}) + f_{k_{m+1}}^-(\bm{x}) + f_{j_{m+1}}^+(\bm{x}))
    = \sum_{m: j_m \neq j_{m+1}}(f_{j_m,k_m}(\bm{x}) - f_{j_{m+1},k_{m+1}}(\bm{x}))\\
    \underset{(d)}{\leq} \sum_{m: j_m \neq j_{m+1}} \left( \max_{g \in A^*_{\varepsilon}(S)}g(\bm{x}) - \min_{g \in A^*_{\varepsilon}(S)}g(\bm{x}) \right)
    \underset{(e)}{\leq} (P^+-1) \left( \max_{g \in A^*_{\varepsilon}(S)}g(\bm{x}) - \min_{g \in A^*_{\varepsilon}(S)}g(\bm{x}) \right)
    \end{gather*}
    $(a)$ uses a telescoping sum, $(b)$ is due to the fact that $f_{j_{m}}^+(\bm{x}) = f_{j_{m+1}}^+(\bm{x})$ when $j_m = j_{m+1}$, $(c)$ derives from $(j_m \neq j_{m+1}) \Rightarrow (k_m = k_{m+1})$, $(d)$ is a result of \cref{lemma4}, and $(e)$ stems from $|\{m: j_m \neq j_{m+1}\}| \leq P^+-1$.
    We note K the set of unique values of $\{k_m, m=1,...,M\}$,  $\underline{m}(p) = \min_{k_m=p}(m)$, and $\overline{m}(p) = \max_{k_m=p}(m)$. We have:
    \begin{gather*}
    f_j^+(\bm{x}) - f_k^+(\bm{x}) 
    =\sum_{p \in K} \left( \sum_{\substack{
      m: j_m \neq j_{m+1}, \\
      k_m=k_{m+1}=p}}  (f_{j_m,k_m}(\bm{x}) - f_{j_{m+1},k_{m+1}}(\bm{x})) \right)\\
    \underset{(f)}{=} \sum_{p \in K} \left(f_{j_{\underline{m}(p)},k_{\underline{m}(p)}}(\bm{x}) - f_{j_{{\overline{m}(p)}},k_{{\overline{m}(p)}}}(\bm{x}) \right)
    \underset{(g)}{\leq} P^- \left( \max_{g \in A^*_{\varepsilon}(S)}g(\bm{x}) - \min_{g \in A^*_{\varepsilon}(S)}g(\bm{x}) \right)
    \end{gather*}
    $(f)$ uses a telescoping sum on $\{\underline{m}(p),...,\overline{m}(p)\}$, and $(g)$ is due to \cref{lemma4} and the fact that $|K| \leq P^-$.
    As a result, 
    $$f_j^+(\bm{x}) - f_k^+(\bm{x}) \leq \min \left( P^+-1,P^- \right) \left( \max_{g \in A^*_\varepsilon(S)}g(\bm{x}) - \min_{g \in A^*_\varepsilon(S)}g(\bm{x}) \right)$$
    Similarly, the lower bound of $f_j^+(\bm{x}) - f_k^+(\bm{x})$ is proven by flipping the inequality in in $(d)$ and $(g)$. 
\end{proof}

\begin{figure}[h!]
    \centering
    \includegraphics[width=0.7\textwidth]{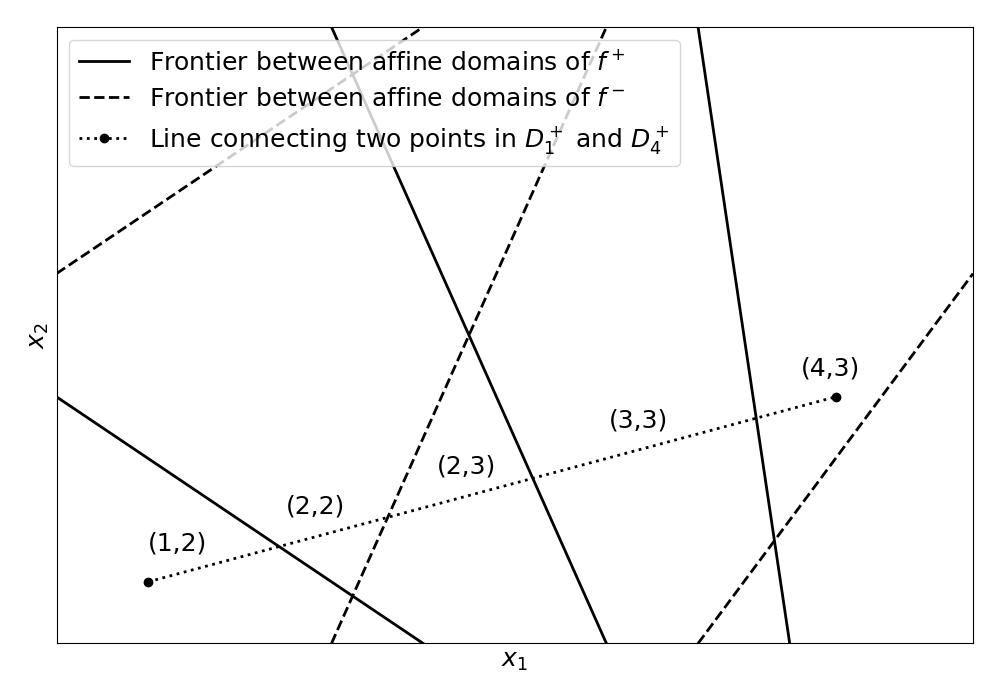}
    \caption{Example of a line connecting two points from distinct affine domains of $f^+$ for $d=2$}
    \label{fig_path_domains}
\end{figure}

\begin{theorem}
\label{theorem_bigM}
    Assuming we only consider well-behaved \gls{cpwl} solutions, the following value is a tight value for the parameter $M_n^c$ in \cref{ConvexEquation}:
    \begin{equation}
    \label{eq_bigM}
        M_n^c = \min \left( P^c-1,P^{\neg{c}} \right) \left( \max_{g \in A^*_\varepsilon(S)}g(\bm{x}_n) - \min_{g \in A^*_\varepsilon(S)}g(\bm{x}_n) \right)
    \end{equation}
\end{theorem}

\begin{proof}[Proof of \cref{theorem_bigM}]
    Let $n \in \{1,...,N\}$, $c \in \{+,-\}$, $j \in \{1,...,P^c\}$. Let $k$ be such that $\bm{x}_n \in D_{k}^c$, i.e., $f^c(\bm{x}_n) = f_{k}^c(\bm{x}_n)$. Owing to \cref{lemma5}, we have:
    \begin{gather*}
    f^c(\bm{x}_n) - f_j^c(\bm{x}_n) 
    = f_{k}^c(\bm{x}_n) - f_{j}^c(\bm{x}_n)
    \leq \min \left( P^c-1,P^{\neg{c}} \right) \left( \max_{g \in A^*_\varepsilon(S)}g(\bm{x}_n) - \min_{g \in A^*_\varepsilon(S)}g(\bm{x}_n) \right)
    \end{gather*}
\end{proof}

Herein, it is assumed that the value of the parameter $M_n^c$ from \cref{ConvexEquation} is the one described in \cref{theorem_bigM}. Note that the big-M values defined in \cref{eq_bigM} are only valid when considering the solution space of well-behaved \gls{cpwl} solutions. This implies that using these values may eliminate some \gls{cpwl} solutions that are not well-behaved from the feasible region, although not necessarily all of them.

\subsection{Bounding the variables}
\label{section_bounding_variables}
We define the following parameters:

$$\underline{a_r} = \min_{g \in A^*_{\varepsilon}(S)}\frac{\partial g}{\partial x_r}, \quad 
 \overline{a_r} = \max_{g \in A^*_{\varepsilon}(S)}\frac{\partial g}{\partial x_r}$$
$$\underline{b} = \min_{g \in A^*_{\varepsilon}(S)}g(\bm{0}), \quad 
 \overline{b} = \max_{g \in A^*_{\varepsilon}(S)}g(\bm{0})$$
$$\overline{a_r'} = \min \left( P^--1,P^+ \right) \left( \overline{a_r} - \underline{a_r} \right)$$
$$\overline{b'} = \min \left( P^--1,P^+ \right) \left( \overline{b} - \underline{b} \right)$$

\begin{lemma}
\label{lemma6}
    Let $f \in CPWL^*(S,\varepsilon,P^+,P^-)$. Let $(j,k) \in \{1,...,P^+\} \times \{1,...,P^-\} : Int(D_{j}^+ \cap D_{k}^-) \neq \varnothing$. Then, we have:
     $$\underline{a_r} \leq a_{j,r}^+ - a_{k,r}^- \leq \overline{a_r}, \quad r \in \{1,...,d\}$$
     $$\underline{b} \leq b_{j}^+ - b_{k}^- \leq \overline{b}$$
\end{lemma}

\begin{proof}[Proof of \cref{lemma6}]
    The partial derivatives of an affine function are equal to its linear coefficients and its value in $\bm{0}$ corresponds to its bias term. In particular, $\frac{\partial f_{j,k}}{\partial x_r}=a_{j,r}^+ - a_{k,r}^-$ and $f_{j,k}(\bm{0}) = b_{j}^+ - b_{k}^-$. Since $f_{j,k} \in A_\varepsilon(S)$, we have  
    $\min_{g \in A_\varepsilon(S)}\frac{\partial g}{\partial x_r} \leq \frac{\partial f_{j,k}}{\partial x_r} \leq \max_{g \in A_\varepsilon(S)}\frac{\partial g}{\partial x_r}$ 
    and $\min_{g \in A_\varepsilon(S)} g(\bm{0}) \leq f_{j,k}(\bm{0}) \leq \max_{g \in A_\varepsilon(S)} g(\bm{0})$. Similar to \cref{lemma5}, we can show that the extrema of $\frac{\partial g}{\partial x_r}$ and $g(\bm{0})$ on $A_{\varepsilon}(S)$ can be found on $A^*_{\varepsilon}(S)$. Let $s \in [S]^{d+1}$. $\max_{g \in A_{\varepsilon,s}}\frac{\partial g}{\partial x_r}$, $-\min_{g \in A_{\varepsilon,s}}\frac{\partial g}{\partial x_r}$, $\max_{g \in A_{\varepsilon,s}}g(\bm{0})$, $-\min_{g \in A_{\varepsilon,s}}g(\bm{0})$, are the optimal objective values of the following LP optimization problem, with $Q = a_r$, $Q = -a_r$, $Q = b$, $Q = -b$, respectively.
    $$\max Q, \quad 
    s.t. \quad -\varepsilon \leq \sum_{r=1}^d a_r x_{i_k,r} + b - z_{i_k} \leq \varepsilon, \quad (\bm{x}_{i_k},z_{i_k})\in s$$
    Therefore, the extrema of $\frac{\partial g}{\partial x_r}$ and $g(\bm{0})$ on $A_{\varepsilon,s}$ can be found on the vertices $A^*_{\varepsilon,s}$. It follows that the extrema on $A_{\varepsilon}(S)$ can be found on $A^*_{\varepsilon}(S)$.
\end{proof}

\begin{lemma}
\label{lemma7}
    Let $f \in CPWL^*(S,\varepsilon,P^+,P^-)$, $(j,k) \in \{1,...,P^-\}^2$. Then,
    $$|a_{j,r}^- - a_{k,r}^-| \leq 
    \overline{a'_r}, \quad \forall r \in \{1,...,d\}$$
    $$|b_j^- - b_k^-| \leq \overline{b'}$$
\end{lemma}

\begin{proof}[Proof of \cref{lemma7}]
    This is demonstrated by using the same sum decomposition as in \cref{lemma5} but for $a_{j,r}^-$, $b_j^-$ instead of $f_j^+(\bm{x})$.
    The result follows from applying \cref{lemma6} to identify a similar inequality as $(d)$ and $(g)$. 
\end{proof}

\begin{theorem}
\label{theorem_bounds}
    For well-behaved \gls{cpwl} solutions, the following equations represent valid upper and lower bounds of the linear variables:
    \begin{equation}
    \label{bounds_fxi}
        z_i - \varepsilon \leq f(\bm{x}_i) \leq z_i + \varepsilon, \quad 
        i \in \{1,...,N\}
    \end{equation}
    \begin{equation}
    \label{bounds_fximinus}
        0 \leq f^-(\bm{x}_i) \leq M_i^-, \quad 
        i \in \{1,...,N\}
    \end{equation}
    \begin{equation}
    \label{bounds_fxiplus}
        z_i - \varepsilon \leq f^+(\bm{x}_i) \leq z_i + \varepsilon + M_i^-, \quad 
        i \in \{1,...,N\}
    \end{equation}
    \begin{equation}
    \label{bounds_ajrm}
        -\overline{a'_r} \leq a_{j,r}^- \leq \overline{a'_r} , \quad 
        j \in \{1,...,P^-\}, \quad r \in \{2,...,d\}
    \end{equation}
    \begin{equation}
    \label{bounds_aj1m}
        0 \leq a_{j,1}^- \leq \overline{a'_1}, \quad j \in \{1,...,P^-\}
    \end{equation}
    \begin{equation}
    \label{bounds_ajrp}
        \underline{a_r} -\overline{a'_r} \leq a_{j,r}^+ \leq \overline{a_r} + \overline{a'_r}, \quad 
        j \in \{1,...,P^+\}, \quad r \in \{2,...,d\}
    \end{equation}
    \begin{equation}
    \label{bounds_aj1p}
        \underline{a_1} \leq a_{j,1}^+ \leq \overline{a_1} + \overline{a'_1}, \quad j \in \{1,...,P^+\}
    \end{equation}
    \begin{equation}
    \label{bounds_bjm}
        -\overline{b'} \leq b_{j}^- \leq \overline{b'} ,
        \quad j \in \{1,...,P^-\}
    \end{equation}
    \begin{equation}
    \label{bounds_bjp}
        \underline{b}-\overline{b'} \leq b_{j}^+ \leq \overline{b} + \overline{b'} ,
        \quad j \in \{1,...,P^+\}
    \end{equation}
\end{theorem}

\begin{proof}[Proof of \cref{theorem_bounds}]
    The bounds from \cref{bounds_fxi} derive from \cref{FittingEquation,ErrorBound}. \Cref{bounds_fximinus} is based on \cref{first_plane_zero}, using $f^-(\bm{x}_i) = \max(f_1^-(\bm{x}_i),f_2^-(\bm{x}_i), ..., f_{P^-}^-(\bm{x}_i)) = \max(0,f_2^-(\bm{x}_i), ..., f_{P^-}^-(\bm{x}_i)) \geq 0$. In addition, $\exists k \in \{1,...,P^-\} : \bm{x}_i \in D_k^-$. Thus, according to \cref{lemma5}, $f^-(\bm{x}_i) = f_k^-(\bm{x}_i) = f_k^-(\bm{x}_i) - f_1^-(\bm{x}_i) \leq M_i^-$. \Cref{bounds_fxiplus} derives from \cref{bounds_fxi,bounds_fximinus} combined with $f^+(\bm{x}_i) = f(\bm{x}_i) + f^-(\bm{x}_i)$. Using \cref{first_plane_zero} and \cref{lemma7}, $|a_{j,r}^-| = |a_{j,r}^- - a_{1,r}^-| \leq \overline{a'_r}$, leading to \cref{bounds_ajrm}. \Cref{bounds_aj1m} arises from \cref{plane_ordering}: $a_{j,1}^- \geq \hdots \geq a_{1,1}^- \geq 0$. Next, combining \cref{bounds_ajrm}, \cref{lemma6}, and $k$ such that $Int(D_j^+ \cap D_k^-) \neq \varnothing$, we deduce: $a_{j,r}^+ = a_{j,r}^+ - a_{k,r}^- + a_{k,r}^-$ and $\underline{a_r} - \overline{a'_r} \leq a_{j,r}^+ - a_{k,r}^- + a_{k,r}^- \leq \overline{a_r} + \overline{a'_r}$, leading to \cref{bounds_ajrp}. \Cref{bounds_aj1p} stems from \cref{bounds_aj1m,bounds_ajrp}. Further, applying \cref{first_plane_zero} and \cref{lemma7}, we deduce $|b_{j}^-| = |b_{j}^- - b_{1}^-| \leq \overline{b}$, resulting in \cref{bounds_bjm}. Finally, \cref{bounds_bjp} stems from \cref{bounds_bjm}, \cref{lemma6}, and $k$ such that $Int(D_j^+ \cap D_k^-) \neq \varnothing$.
\end{proof}

Note that the bounding values from \cref{bounds_fximinus,bounds_fxiplus,bounds_ajrm,bounds_aj1m,bounds_ajrp,bounds_aj1p,bounds_bjm,bounds_bjp} are only valid when considering the solution space of well-behaved \gls{cpwl} solutions. This implies that using these bounding values may eliminate some \gls{cpwl} solutions that are not well-behaved from the feasible region, although not necessarily all of them. In addition, \cref{bounds_ajrm,bounds_aj1m,bounds_ajrp,bounds_aj1p,bounds_bjm,bounds_bjp} apply owing to \cref{first_plane_zero,plane_ordering}. This implies that using these values may eliminate some DC representations of the \gls{cpwl} solutions but not the \gls{cpwl} solutions themselves.

\begin{proposition}
\label{time_complexity}
    The time complexity for calculating the bounds and big-M parameters is $O(2^{d+1}N^{d+2}d)$.
\end{proposition}

\begin{proof}[Proof of \cref{time_complexity}]
    The proof of this proposition is based upon the number of elements in $A_{\varepsilon}^*(S)$ and the time complexity of inverting a matrix in $M_{d+1}(\mathbb{R})$. A more detailed proof is provided in \cref{proof_time_complexity}.
\end{proof}

\Cref{table_summary} gives a summary of the tightening strategies, together with their impact on the feasible region and the preprocessing involved. The two sets of equations marked by ``\dag'' are alternative to one another, the first set being dominated by the second set but involving less constraints and variables than the second set. The symbol ``*'' in the number of additional constraints indicates the number of additional variable bounds.

\begin{table}
\caption{Summary of tightening strategies}
\centering
\scriptsize
\label{table_summary}
{\begin{tabular}[c]{m{0.3cm} m{2.1cm} m{2.6cm} m{1.4cm} m{2.5cm} m{2.1cm} m{1.8cm}}
  \hline
  \raggedright \textbf{Eq.} 
  & \raggedright \textbf{Description}
  % & \raggedright For well-behaved solutions only? 
  & \raggedright \textbf{\# of new constraints or *variable bounds} 
  & \raggedright \textbf{\# of new variables}
  & \raggedright \textbf{Impact on search region}
  & \raggedright \textbf{Preprocessing}
  & \raggedright \textbf{Time complexity}
  \arraybackslash \\
  \hline
  \eqref{first_plane_zero} 
  & \raggedright Fix one affine piece 
  % &  
  & \makecell[cl]{0 \\ *$d+1$}
  & 0 
  & \raggedright Eliminate $d+1$ dimensions
  & None
  &  \\
  \hline
  \eqref{plane_ordering} 
  & \raggedright Sort the affine pieces 
  % &  
  & $P^++P^--2$
  & 0 
  & \raggedright Reduce the feasible region by a factor of $P^+!P^-!$
  & None 
  &  \\
  \hline
  \eqref{number_points_per_plane}$^{\dag}$
  & \raggedright Impose $d+1$ points per affine piece of $f^c$
  % & \checkmark
  & $P^++P^-$ 
  & 0 
  & \raggedright Eliminate some non-well-behaved CPWL solutions
  & None 
  &  \\
  \hline
    \eqref{beta_delta_plus}-\eqref{gamma_bounds}$^{\dag}$
  & \raggedright Impose $d+1$ points per affine piece of $f$ 
  % & \checkmark 
  & \raggedright $P^+P^-(4N+1)$ 
  & \raggedright $P^+P^-(N+1)$ 
  & \raggedright Eliminate all non-well-behaved CPWL solutions 
  & None 
  &  \\
  \hline
   \eqref{eq_bigM}
   & \raggedright Tighten the big-M parameters 
   % & \checkmark 
   & 0 
   & 0 
   & \raggedright Eliminate some non-well-behaved CPWL solutions 
   & \raggedright Compute all $g \in A_\varepsilon^*(S)$, evaluate $g$ in $N$ points, compute extrema for each point 
   & $O(2^{d+1}N^{d+2}d)$ \\
  \hline
  \eqref{bounds_fxi}-\eqref{bounds_bjp}
  & Bound the variables
  % & \checkmark 
  & \makecell[cl]{0 \\ *$3N + $\\$(P^++P^-)(d+1)$}
  & 0 
  & \raggedright Bound the search space in all variable dimensions
  & \raggedright For all $g \in A_\varepsilon^*(S)$, evaluate the extrema of the coefficients of $g$  
  & $O(2^{d+1}N^{d+1}d)$ \\
  \hline
\end{tabular}}
\end{table}

\begin{remark}
Note that applying all equations presented in \cref{tighteningsection} is a valid tightening of the \gls{milp} problem to the feasible region of the well-behaved \gls{cpwl} solutions. As a result, applying any subset of equations from \cref{tighteningsection} is also a valid tightening. For example, although the variable bounds from \cref{section_bounding_variables} derive from fixing $f_1^-=0$ and sorting the affine pieces of $f^+$ and $f^-$, applying the variable bounds without sorting or fixing any affine piece is still a valid tightening of the \gls{milp} problem. As a matter of fact, any combination of the tightening strategies described here is a valid tightening. This fact is leveraged in \cref{section_experiments} to identify effective combinations of tightening strategies.
\end{remark}

\section{Computational experiments}
\label{section_experiments}

This section illustrates the computational performances of the tightening strategies. In section \cref{dataset_description}, the data sets used to assess the efficiency of the tightening strategies are described. In section \cref{combination_strategies}, different combinations of the tightening strategies are evaluated to identify the ones that have the most impact on the solution time. In addition, the computational performance of the tightening procedure is assessed. 

The algorithms are implemented in Python (3.12.7) and the \gls{milp} problems are solved with the Gurobi solver (12.0.0) using the gurobipy package. Models are run on an Intel 2.3-GHz machine with 16 cores and 32 GB of RAM. For all \gls{milp} problems, the relative optimality gap (MIPGap) is set to $10^{-6}$. In addition, the primal feasibility tolerance (FeasibilityTol) and integer feasibility tolerance (IntFeasTol) are both set to their minimum value of $10^{-9}$ to mitigate numerical errors.

The objective function used for all \gls{milp} problems is the maximum error, \textit{i.e.}, \cref{eq_max_error}.

\subsection{Description of the data sets}
\label{dataset_description}
To assess the efficiency of the proposed method, six data sets are used. The data sets are summarized in \cref{data_description}. Four of the data sets are derived from mathematical functions and two data sets are based on real-world data. Data sets based on mathematical functions are generated by randomly selecting a certain number of $\bm{x}$ points in a given domain of $\mathbb{R}^d$. The first four data sets from \cref{data_description} are two-dimensional and are illustrated by the red dots in \cref{fig_case_studies}. The last two data sets are three-dimensional. The third data set represents the water-to-power conversion factor ($z$) at the Crystal hydropower plant based on the forebay elevation ($x_1$) and average daily water release ($x_2$) from the years 2014-2024 (\cite{bureau_of_reclamation_hdb_2024}). The fourth data set represents the discharge temperature of a gas compressor ($z$) based on the volumetric flow rate ($x_1$) and rotation speed ($x_2$) (\cite{marfatia_data-driven_2022}). Data sets were rescaled in all dimensions to prevent numerical issues during the \gls{milp} solving process.

\begin{table}[h!]
    \centering
    \caption{Data sets used for numerical experiments}
    \label{data_description}
    \begin{tabular}{@{}lrr@{}}
        \toprule
        \textbf{Data set} & \textbf{Number of data points} & \textbf{Dimension $d$} \\
        \midrule
        $z = x_2 \sin(x_1)$, $(x_1, x_2) \in [0, \pi] \times [0, 1]$ & 121 & 2 \\
        $z = x_1^2 - x_2^2$, $(x_1, x_2) \in [-1, 1]^2$ & 64 & 2 \\
        Crystal power plant \cite{bureau_of_reclamation_hdb_2024} & 116 & 2 \\
        Gas compressor \cite{marfatia_data-driven_2022} & 102 & 2 \\
        $z = x_1^2 + x_2^2 + x_3^2$, $(x_1, x_2, x_3) \in [0, 1]^3$ & 64 & 3 \\
        $z = x_1 x_2 x_3$, $(x_1, x_2, x_3) \in [0, 1]^3$ & 64 & 3 \\
        \bottomrule
    \end{tabular}
\end{table}

\begin{figure}[h!]
    \centering
    \includegraphics[width=1.0\textwidth]{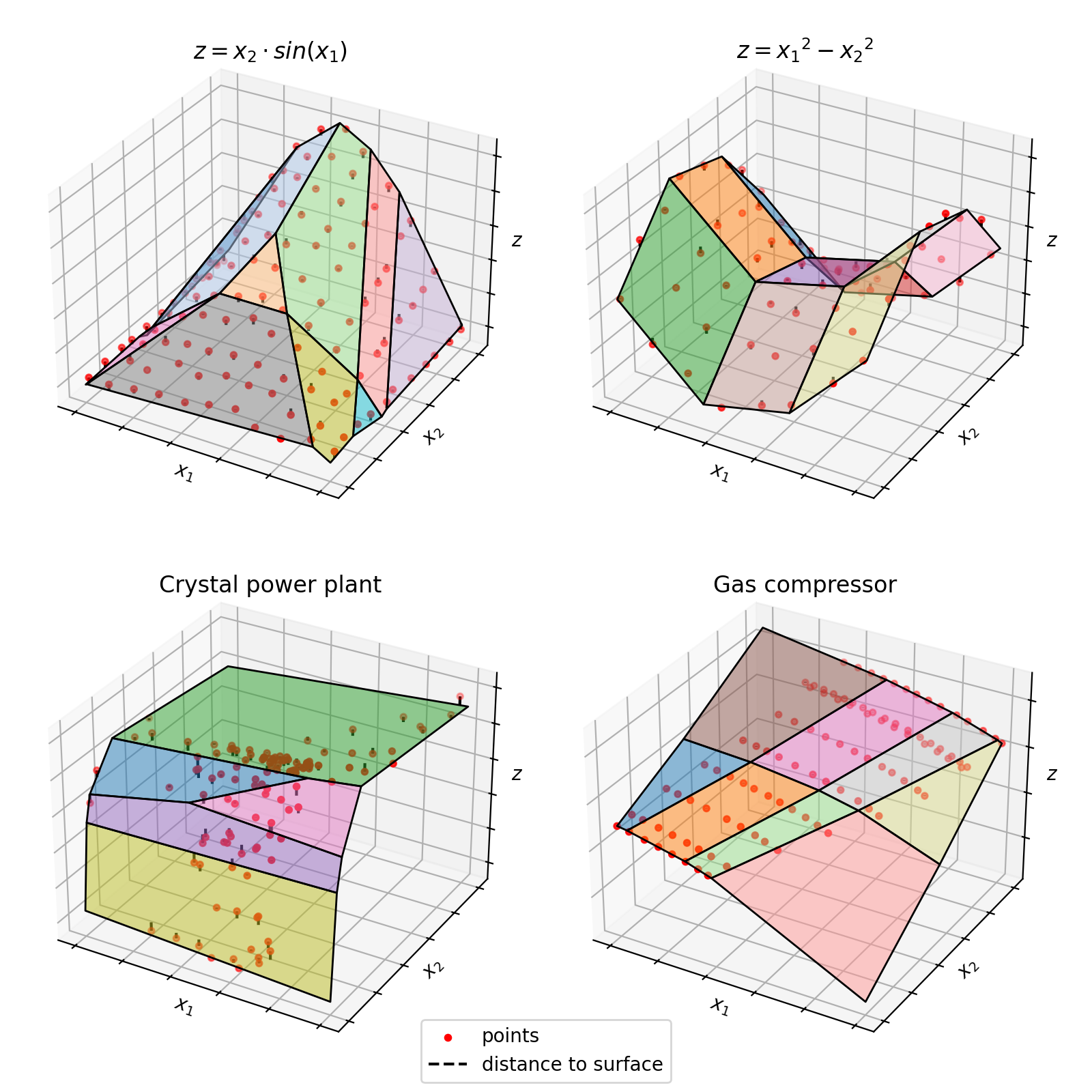}
    \caption{Optimal CPWL fitting for the two-dimensional data sets}
    \label{fig_case_studies}
\end{figure}

\begin{figure}[h!]
    \centering
    \includegraphics[width=1.0\textwidth]{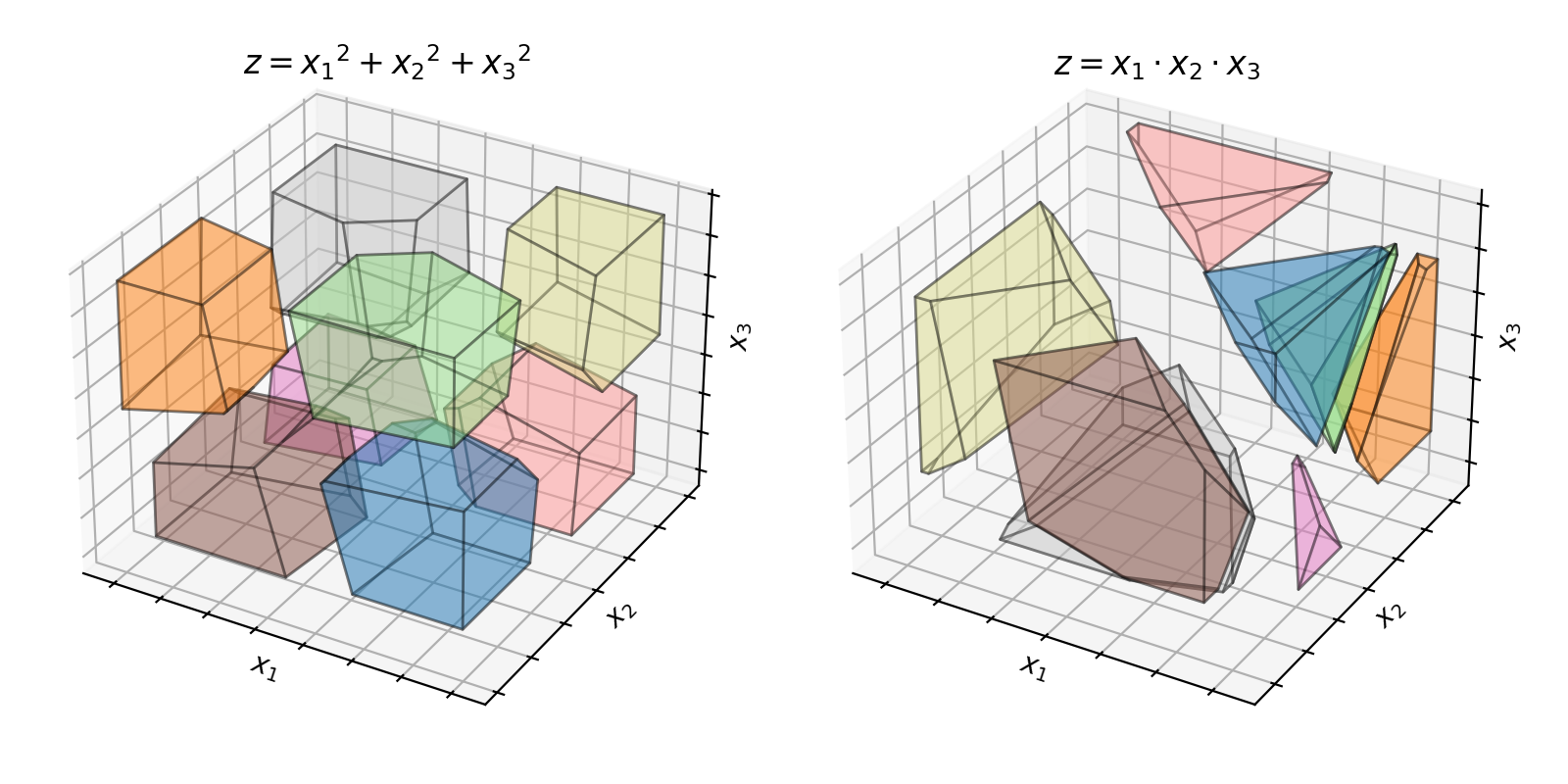}
    \caption{Exploded view of the affine domains for the three-dimensional data sets}
    \label{fig_case_studies2}
\end{figure}

\subsection{Evaluating different combinations of tightening strategies}
\label{combination_strategies}

There are several ways to apply the tightening strategies to the \gls{milp} problem:
\begin{itemize}
    \item The first affine piece of $f^-$ may or may not be set to $0$ (\cref{first_plane_zero}).
    \item The affine pieces of $f^+$ and $f^-$ may or may not be sorted (\cref{plane_ordering}).
    \item We can impose $d+1$ points either for each affine piece of $f^+$ and $f^-$ (\cref{number_points_per_plane}), for each affine piece of $f$ (\cref{beta_delta_plus,beta_delta_minus,beta_delta_plus_minus,gamma_beta,sum_beta_gamma}), or for no affine piece.
    \item For the big-M parameters, we can either use an indicator constraint (\cref{ConvexEquationIndicator}), a reasonably large (``default") big-M value, or the tight big-M values (\cref{eq_bigM}).
    \item The variables may or may not be bounded (\cref{bounds_aj1m,bounds_ajrm,bounds_bjm,bounds_aj1p,bounds_ajrp,bounds_bjp,bounds_fxi,bounds_fximinus,bounds_fxiplus}).
    
\end{itemize}

In total, there are up to 72 combinations of tightening strategies. Due to this large number, only eleven combinations were selected. The selected combinations, referred to as ``C1'' to ``C11'', are described in \cref{tight_combinations}. The ``default" big-M value is set equal to the rounded-up maximum value of $M^c_i$, \textit{e.g.}, 700 if the maximum value is 632.8. For each data set, the eleven combinations of tightening strategies are applied to the \gls{milp} problem before solving it, and the solution time is measured for each combination. The same number $P^+$ and $P^-$ of candidate affine pieces and the same upper bound $\varepsilon$ for the maximum error are used across all combinations. The time limit to solve the \gls{milp} problem is set to 7200 seconds. Numerical results are summarized in \cref{MILP_results}. The table describes the computation times for each data set, including the computation time of the preprocessing step and the \gls{milp} solving time of each combination of tightening strategies. The symbol ``*'' indicates that the time limit was reached before meeting the target optimality gap. Bold numbers indicate the two shortest solving times for each data set across all combinations.

\begin{table}[h!]
    \centering
    \footnotesize
    \begin{threeparttable}
        \caption{Combinations of tightening strategies assessed.}
        \label{tight_combinations}
        \begin{tabularx}{\textwidth}{@{}l >{\RaggedRight}X >{\RaggedRight}X >{\RaggedRight}X >{\RaggedRight}X >{\RaggedRight}X@{}}
            \toprule
            \textbf{Combination} & \textbf{Fix one affine piece} & \textbf{Sort the affine pieces} & \textbf{Impose $d+1$ points per affine piece} & \textbf{Big-M parameters} & \textbf{Bounded variables} \\
            \midrule
            C1 & & & & Indicator & \\
            \midrule
            C2 & & & & Default & \\
            \midrule
            C3 & & & & Tight & \\
            \midrule
            C4 & \checkmark & & & Tight & \\
            \midrule
            C5 & \checkmark & \checkmark & & Tight & \\
            \midrule
            C6 & \checkmark & \checkmark & & Tight & \checkmark \\
            \midrule
            C7 & \checkmark & \checkmark & of $f^+$ and $f^-$ & Tight & \checkmark \\
            \midrule
            C8 & \checkmark & \checkmark & of $f$ & Tight & \checkmark \\
            \midrule
            C9 & & & of $f^+$ and $f^-$ & Tight & \checkmark \\
            \midrule
            C10 & & & of $f$ & Tight & \checkmark \\
            \midrule
            C11 & \checkmark & & of $f^+$ and $f^-$ & Tight & \checkmark \\
            \bottomrule
        \end{tabularx}
        % \begin{tablenotes}
        %     \item A blank cell indicates that the tightening strategy does not apply to the combination.
        % \end{tablenotes}
    \end{threeparttable}
\end{table}

\begin{table}
\caption{Summary of computation times for each data set and combination. \label{MILP_results}}
\centering
\scriptsize
{\begin{tabular}[c]{m{1.6cm} m{0.8cm} m{0.3cm} m{1.2cm} m{0.5cm} m{0.5cm} m{0.5cm} m{0.5cm} m{0.5cm} m{0.5cm} m{0.5cm} m{0.5cm} m{0.5cm} m{0.5cm} m{0.5cm} m{0.5cm}}
  \hline
  \multirow{2}{4em}{\textbf{Data set}}
  & \multirow{2}{4em}{$\bm{P^+}, \bm{P^-}$}
  & \multirow{2}{4em}{$\bm{\varepsilon}$}
  % & \multirow{2}{4em}{Default big-M}
  & \multicolumn{12}{c}{\textbf{Computation time (s)}} \\
  & & & \raggedright \textbf{Prepro- cessing} & \textbf{C1} & \textbf{C2} & \textbf{C3} & \textbf{C4} & \textbf{C5} & \textbf{C6} & \textbf{C7} & \textbf{C8} & \textbf{C9} & \textbf{C10} & \textbf{C11} \\
  \hline
  \raggedright $z = x_2 \cdot sin(x_1)$
  & 2, 6  
  & 0.2 
  & 9   
  & 7200*      
  & 7200*        
  & 7200*        
  & 7200*        
  & 7200*        
  & 7200*        
  & 7200*        
  & 7200*         
  & \textbf{1877}        
  & 7200*        
  & \textbf{2554}  \\
  \hline
  $z = {x_1}^2 - {x_2}^2$  
  & 3, 3 
  & 0.1 
  & 1 
  & 7200* 
  & 576.7 
  & \textbf{112.8} 
  & 116.9
  & 507.8
  & 382.2
  & 406.4
  & 2752
  & \textbf{98.3}
  & 516.6 
  & 200.2 \\
  \hline
  Crystal power plant
  & 1, 5
  & 0.2 
  & 8 
  & 197.6
  & \textbf{8.7} 
  & 9.9 
  & \textbf{8.6}
  & 72.0
  & 42.0
  & 32.3
  & 31.5
  & 11.4
  & 9.7
  & 20.8 \\
  \hline
    \raggedright Gas compressor
  & 2, 4
  & 0.2 
  & 5 
  & 7200* 
  & 7200* 
  & 1584 
  & 2090
  & 847.4
  & 1356
  & 944.9
  & 7200*
  & \textbf{457.2}
  & 3720 
  & \textbf{303.2} \\
  \hline
  $z = {x_1}^2 + {x_2}^2 + {x_3}^2$
  & 8, 1
  & 0.1 
  & 37 
  & 7200* 
  & 3402 
  & 1441
  & \textbf{111.0}
  & 7200*
  & 7200*
  & 7200*
  & 7200*
  & 531.1
  & 478.3
  & \textbf{159.5} \\
  \hline
  $z = {x_1} \cdot {x_2} \cdot {x_3}$
  & 2, 4 
  & 0.1 
  & 38 
  & 7200* 
  & 2740 
  & \textbf{460.3}
  & 558.3
  & 2496
  & 1968
  & 974.2
  & 7200*
  & 698.1
  & 4390
  & \textbf{438.1} \\
  \hline
\end{tabular}}
\end{table}

For each data set, the optimal objective value of the \gls{milp} problem is identical across all combinations. However, solution times noticeably vary across combinations. Results show that there is a clear benefit of using tight big-M parameter values (C3) as opposed to using a default big-M parameter value (C2) or an indicator constraint (C1). Fixing the first affine piece of $f^-$ (C4) without imposing $d+1$ points per affine piece does not seem to have a significant additional benefit, except for the case ``$z={x_1}^2+{x_2}^2+{x_3}^2$". Although it reduces the feasible region, sorting the affine pieces of $f^-$ and $f^+$ has a negative impact on the solution time (C5 to C8). This might be due to the fact that these additional constraints increase the solution time of the relaxed LP problem in each node of the \gls{milp} search tree without improving the tree search itself, \textit{i.e.}, without constraining the binary variables. Due to its negative impact, the strategy of sorting the affine pieces is discarded from the rest of the combinations. Imposing $d+1$ points per affine piece of $f^+$ and $f^-$ and bounding the variables (C9 and C11) has a clear positive impact on the solution time. This positive effect can be explained by the reduction in the feasible region of the binary variables $\delta_{i,j}^c$ using few additional constraints ($P^+ + P^-$). Although this feasible region is further reduced when imposing $d+1$ points per affine piece of $f$ (C10), this constraint has a negative impact on the solution time. This might be due to the larger number of additional variables and constraints required to model the affine pieces of $f$. The only difference between combinations C9 and C11 is whether the first affine piece of $f^-$ is being fixed. Numerical results show that adding this strategy may either improve or worsen the solution time depending on the data set, \textit{i.e.}, there is not a clear systematic benefit in applying this strategy. When comparing the best solution time of C9 and C11 to the solution time of C2, the reduction factor in solving time ranges from 4 to 23, except for the Crystal power plant case. For the worst solution time, the reduction factor ranges from 3 to 16. Overall, the following combination consistently performs among the best ones: imposing $d+1$ points for each affine piece of $f^+$ and $f^-$, (\cref{number_points_per_plane}), tightening the big-M parameters (\cref{eq_bigM}), fixing the first affine piece of $f^-$ (\cref{first_plane_zero}), and bounding all variables (\cref{bounds_aj1m,bounds_ajrm,bounds_bjm,bounds_aj1p,bounds_ajrp,bounds_bjp,bounds_fxi,bounds_fximinus,bounds_fxiplus}).

The computation time of the preprocessing step, used to calculate the big-M parameters and variable bounds, is also included in \cref{MILP_results}. The computation time of this step is negligible for two-dimensional data sets. However, the computation time becomes significant for three-dimensional data sets. This is due to the fact that the time complexity of the preprocessing steps increases exponentially with the dimension of the data set (\cref{table_summary}). Yet, the combined execution time of preprocessing and solving the tightened \gls{milp} problem is still shorter than the time required to solve the non-tightened \gls{milp} problem, demonstrating the efficiency of the tightening approach in reducing overall computational effort.

\section{Conclusions}
%%\label{}
In this paper, we formalize the concept of well-behaved \gls{cpwl} interpolations, a class of \gls{cpwl} interpolations in which each affine piece interpolates a number of points greater than the dimension of the interpolated data set. Next, we demonstrate that any \gls{cpwl} interpolation has a well-behaved version. Then, we introduce several tightening strategies aimed at improving the solution time of the \gls{milp} formulation of the \gls{cpwl}-fitting problem for data sets in general dimensions. Some of the tightening strategies leverage the fact that any \gls{cpwl} interpolation has a well-behaved version. In addition, we analyze each tightening strategy in terms of additional constraints, additional variables, impact on the feasible region, and time complexity of the preprocessing step. Then, we identify the combinations of tightening strategies that have the most impact on the solution time of the \gls{milp} problem. Experimental results show that the tightening procedure significantly reduces the solution time of the \gls{milp} problem in most cases. However, the solution time reduction factor can vary significantly depending on the size and dimension of the data set, with values ranging from 3 to 23 across five of the six case studies. In addition, the computation time of the preprocessing step rapidly increases with the dimension of the data set, making the tightening procedure computationally challenging for high-dimensional data sets ($d>3$). This issue can be partially addressed using parallelization, as the preprocessing step is inherently suited for parallel processing. 

The theoretical work presented here applies to data points in general position (with any subset of $d+1$ points affinely independent in $\mathbb{R}^d$). There is a need to address the special case where the points are not in general position, \textit{e.g.}, if the points are located on a lattice. Research work should also focus on improving the time complexity of the preprocessing algorithms for high dimensions. In addition, future research should aim to conduct a comprehensive comparison between the \gls{milp} and \gls{nn} approaches with respect to computation time, number of affine pieces, and approximation error. 

\section*{Acknowledgements}

This work was authored for the Department of Energy (DOE) Office of Energy Efficiency and Renewable Energy by Argonne National Laboratory, operated by UChicago Argonne LLC under contract number DE-AC02-06CH11357. This study was supported by the HydroWIRES Initiative of DOE’s Water Power Technologies Office.

\bibliographystyle{alpha} 
\bibliography{references.bib}

\begin{thebibliography}{{Gur}24b}

\bibitem[BT97]{bertsimas_introduction_1997}
Dimitris Bertsimas and John Tsitsiklis.
\newblock {\em Introduction to {Linear} {Optimization}}.
\newblock Athena Scientific, 1997.

\bibitem[{Bur}24]{bureau_of_reclamation_hdb_2024}
{Bureau of Reclamation}.
\newblock {HDB} {Data} {Service}, 2024.

\bibitem[BV04]{boyd_convex_2004}
Stephen Boyd and Lieven Vandenberghe.
\newblock {\em Convex {Optimization}}.
\newblock Cambridge University Press, 2004.

\bibitem[CGR23]{chen_improved_2023}
Kuan-Lin Chen, Harinath Garudadri, and Bhaskar~D. Rao.
\newblock Improved {Bounds} on {Neural} {Complexity} for {Representing} {Piecewise} {Linear} {Functions}, January 2023.
\newblock arXiv:2210.07236 [cs].

\bibitem[Dav75]{davis_interpolation_1975}
Philip~J. Davis.
\newblock {\em Interpolation and {Approximation}}.
\newblock Courier Corporation, January 1975.
\newblock Google-Books-ID: 2PaJAwAAQBAJ.

\bibitem[DHP21]{devore_neural_2021}
Ronald DeVore, Boris Hanin, and Guergana Petrova.
\newblock Neural network approximation.
\newblock {\em Acta Numerica}, 30:327--444, May 2021.

\bibitem[DLM10]{dambrosio_piecewise_2010}
Claudia D’Ambrosio, Andrea Lodi, and Silvano Martello.
\newblock Piecewise linear approximation of functions of two variables in {MILP} models.
\newblock {\em Operations Research Letters}, 38(1):39--46, January 2010.

\bibitem[DN22]{duguet_piecewise_2022}
Aloïs Duguet and Sandra~Ulrich Ngueveu.
\newblock Piecewise {Linearization} of {Bivariate} {Nonlinear} {Functions}: {Minimizing} the {Number} of {Pieces} {Under} a {Bounded} {Approximation} {Error}.
\newblock In Ivana Ljubić, Francisco Barahona, Santanu~S. Dey, and A.~Ridha Mahjoub, editors, {\em Combinatorial {Optimization}}, pages 117--129, Cham, 2022. Springer International Publishing.

\bibitem[EOS86]{edelsbrunner_constructing_1986}
H.~Edelsbrunner, J.~O’Rourke, and R.~Seidel.
\newblock Constructing {Arrangements} of {Lines} and {Hyperplanes} with {Applications}.
\newblock {\em SIAM Journal on Computing}, 15(2):341--363, May 1986.
\newblock Publisher: Society for Industrial and Applied Mathematics.

\bibitem[FSB10]{frenzen_number_2010}
C.~L. Frenzen, Tsutomu Sasao, and Jon~T. Butler.
\newblock On the number of segments needed in a piecewise linear approximation.
\newblock {\em Journal of Computational and Applied Mathematics}, 234(2):437--446, May 2010.

\bibitem[GMMS12]{geisler_using_2012}
Björn Geißler, Alexander Martin, Antonio Morsi, and Lars Schewe.
\newblock Using {Piecewise} {Linear} {Functions} for {Solving} {MINLPs}.
\newblock In Jon Lee and Sven Leyffer, editors, {\em Mixed {Integer} {Nonlinear} {Programming}}, pages 287--314, New York, NY, 2012. Springer.

\bibitem[{Gur}24a]{gurobi_optimization_llc_gurobi_2024}
{Gurobi Optimization LLC}.
\newblock Gurobi {Optimizer} {Reference} {Manual}, 2024.

\bibitem[{Gur}24b]{gurobi_optimization_llc_modeladdgenconstrindicator_2024}
{Gurobi Optimization LLC}.
\newblock Model.{addGenConstrIndicator}(), 2024.

\bibitem[HA96]{hughes_simplexity_1996}
Robert~B. Hughes and Michael~R. Anderson.
\newblock Simplexity of the cube.
\newblock {\em Discrete Mathematics}, 158(1):99--150, October 1996.

\bibitem[Hua20]{huang_relu_2020}
Changcun Huang.
\newblock {ReLU} {Networks} {Are} {Universal} {Approximators} via {Piecewise} {Linear} or {Constant} {Functions}.
\newblock {\em Neural Computation}, 32(11):2249--2278, November 2020.

\bibitem[KL21]{kazda_nonconvex_2021}
Kody Kazda and Xiang Li.
\newblock Nonconvex multivariate piecewise-linear fitting using the difference-of-convex representation.
\newblock {\em Computers \& Chemical Engineering}, 150:107310, July 2021.

\bibitem[KL24]{kazda_linear_2024}
Kody Kazda and Xiang Li.
\newblock A linear programming approach to difference-of-convex piecewise linear approximation.
\newblock {\em European Journal of Operational Research}, 312(2):493--511, January 2024.

\bibitem[KM20]{kong_derivation_2020}
Lingxun Kong and Christos~T Maravelias.
\newblock On the {Derivation} of {Continuous} {Piecewise} {Linear} {Approximating} {Functions}.
\newblock {\em INFORMS Journal on Computing}, 32(3):531--546, July 2020.
\newblock Publisher: INFORMS.

\bibitem[KS87]{kripfganz_piecewise_1987}
Anita Kripfganz and R.~Schulze.
\newblock Piecewise affine functions as a difference of two convex functions.
\newblock {\em Optimization}, 18(1):23--29, January 1987.
\newblock Publisher: Taylor \& Francis \_eprint: https://doi.org/10.1080/02331938708843210.

\bibitem[MF10]{misener_piecewise-linear_2010}
R.~Misener and C.~A. Floudas.
\newblock Piecewise-{Linear} {Approximations} of {Multidimensional} {Functions}.
\newblock {\em Journal of Optimization Theory and Applications}, 145(1):120--147, April 2010.

\bibitem[ML22]{marfatia_data-driven_2022}
Zaid Marfatia and Xiang Li.
\newblock Data-{Driven} {Natural} {Gas} {Compressor} {Models} for {Gas} {Transport} {Network} {Optimization}.
\newblock {\em Digital Chemical Engineering}, 3:100030, June 2022.

\bibitem[Plo24]{ploussard_piecewise_2024}
Quentin Ploussard.
\newblock Piecewise linear approximation with minimum number of linear segments and minimum error: {A} fast approach to tighten and warm start the hierarchical mixed integer formulation.
\newblock {\em European Journal of Operational Research}, 315(1):50--62, May 2024.

\bibitem[RK15]{rebennack_continuous_2015}
Steffen Rebennack and Josef Kallrath.
\newblock Continuous {Piecewise} {Linear} {Delta}-{Approximations} for {Bivariate} and {Multivariate} {Functions}.
\newblock {\em Journal of Optimization Theory and Applications}, 167(1):102--117, October 2015.

\bibitem[RK20]{rebennack_piecewise_2020}
Steffen Rebennack and Vitaliy Krasko.
\newblock Piecewise {Linear} {Function} {Fitting} via {Mixed}-{Integer} {Linear} {Programming}.
\newblock {\em INFORMS Journal on Computing}, 32(2):507--530, April 2020.
\newblock Publisher: INFORMS.

\bibitem[TV12]{toriello_fitting_2012}
Alejandro Toriello and Juan~Pablo Vielma.
\newblock Fitting piecewise linear continuous functions.
\newblock {\em European Journal of Operational Research}, 219(1):86--95, May 2012.

\bibitem[VAN10]{vielma_mixed-integer_2010}
Juan~Pablo Vielma, Shabbir Ahmed, and George Nemhauser.
\newblock Mixed-{Integer} {Models} for {Nonseparable} {Piecewise}-{Linear} {Optimization}: {Unifying} {Framework} and {Extensions}.
\newblock {\em Operations Research}, 58(2):303--315, April 2010.
\newblock Publisher: INFORMS.

\bibitem[Vie15]{vielma_mixed_2015}
Juan~Pablo Vielma.
\newblock Mixed {Integer} {Linear} {Programming} {Formulation} {Techniques}.
\newblock {\em SIAM Review}, 57(1):3--57, January 2015.
\newblock Publisher: Society for Industrial and Applied Mathematics.

\bibitem[WR22]{warwicker_comparison_2022}
John~Alasdair Warwicker and Steffen Rebennack.
\newblock A {Comparison} of {Two} {Mixed}-{Integer} {Linear} {Programs} for {Piecewise} {Linear} {Function} {Fitting}.
\newblock {\em INFORMS Journal on Computing}, 34(2):1042--1047, March 2022.
\newblock Publisher: INFORMS.

\bibitem[WR23]{warwicker_generating_2023}
John~Alasdair Warwicker and Steffen Rebennack.
\newblock Generating optimal robust continuous piecewise linear regression with outliers through combinatorial {Benders} decomposition.
\newblock {\em IISE Transactions}, 55(8):755--767, August 2023.
\newblock Publisher: Taylor \& Francis \_eprint: https://doi.org/10.1080/24725854.2022.2107249.

\end{thebibliography}

\newpage

\begin{quote}
    \footnotesize
    The submitted manuscript has been created by UChicago Argonne, LLC, Operator of Argonne National Laboratory (“Argonne”). Argonne, a U.S. Department of Energy Office of Science laboratory, is operated under Contract No. DE-AC02-06CH11357. The U.S. Government retains for itself, and others acting on its behalf, a paid-up nonexclusive, irrevocable worldwide license in said article to reproduce, prepare derivative works, distribute copies to the public, and perform publicly and display publicly, by or on behalf of the Government.  The Department of Energy will provide public access to these results of federally sponsored research in accordance with the DOE Public Access Plan. http://energy.gov/downloads/doe-public-access-plan 
\end{quote}

\newpage

\section*{Appendix}

\appendix

\section{Alternative objective functions for $MILP1(Q)$}
\label{alternative_objectives}

The following objective functions $Q$ may be considered when solving the problem $MILP1(Q)$:

\begin{equation}
\label{eq_nber_planes}
    Q = \sum_{j=1}^{P^+} \sum_{k=1}^{P^-} \gamma_{j,k}
\end{equation}
\begin{equation}
\label{eq_nber_planes_pm}
    Q = \sum_{j=1}^{P^c} \alpha_j^c
\end{equation}
\begin{equation}
\label{eq_nber_planes_p}
    \alpha_j^+ \geq \gamma_{j,k}, \quad j \in \{1,...,P^+\}, \quad k \in \{1,...,P^-\}
\end{equation}
\begin{equation}
\label{eq_nber_planes_m}
    \alpha_k^- \geq \gamma_{j,k}, \quad j \in \{1,...,P^+\}, \quad k \in \{1,...,P^-\}
\end{equation}
\begin{equation}
\label{eq_hierarchical}
    Q = Q_1 + \frac{1}{2\varepsilon}Q_2
\end{equation}

Objective \eqref{eq_nber_planes} represents the number of affine pieces of the \gls{cpwl} solution. Objective \eqref{eq_nber_planes_pm} represents the number of affine pieces of $f^c$, and requires using the additional variables $\alpha_j^c$ and \cref{eq_nber_planes_p,eq_nber_planes_m}. Objective \eqref{eq_hierarchical} formulates a hierarchical optimization problem, with $Q_1$ being the number of affine pieces of $f$, $f^+$, or $f^-$, and $Q_2$ being the average or maximum error. Authors in \cite{ploussard_piecewise_2024} show that the choice of coefficients in \eqref{eq_hierarchical} results in a hierarchical optimization problem where $Q_1$ and $Q_2$ are co-minimized but minimizing $Q_1$ takes priority over minimizing $Q_2$.

\section{Additional tightening strategies}
\label{section_other_strategies}
The following tightening strategies complement the six tightening strategies described in the main paper. However, their impact on the feasible region is unclear and they do not seem to have a significant impact on the solution time.

\subsection{Using the convexity of $f^c$}
\label{section_convexity_function}
\begin{fact}
\label{fact5}
    Let $f:D \rightarrow \mathbb{R}$ be a convex function that is piecewise differentiable. Let $\bm{x}_1, \bm{x}_2 \in D$. Let $\bm{g}_1$ and  $\bm{g}_2$ be subgradients of $\bm{x}_1$ and $\bm{x}_2$, respectively. Then, the following inequality holds: $(\bm{x}_2-\bm{x}_1)^T(\bm{g}_2 - \bm{g}_1) \geq 0$.
\end{fact}
This inequality is a generalization of the monotonicity condition of the gradient of differentiable convex functions to piecewise differentiable functions (\cite{boyd_convex_2004}).
\begin{theorem}
\label{theorem_monoticity}
\textit{The following equation serves as a valid tightening constraint of $MILP1(Q)$}: 
\begin{equation}
\label{eq_monoticity}
\begin{aligned}
    (\bm{a}_j^c - \bm{a}_k^c)^T(\bm{x}_p - \bm{x}_q) &\geq -M_{p,q}^c \left(2 - \delta_{p,j}^c - \delta_{q,k}^c\right), \\
    \text{where } \quad p,q &\in \{1,\dots,N\}, \quad j,k \in \{1,\dots,P^c\}, \quad c \in \{+,-\},
    \quad p &\neq q, \quad j < k
\end{aligned}
\end{equation}
\end{theorem}

\begin{proof}[Proof of \cref{theorem_monoticity}]
    Let $f$ be a \gls{cpwl} solution of $MILP1(Q)$ and $f= f^+ - f^-$ be a DC representation. $f^+$ and $f^-$ are convex piecewise differentiable functions. Let $\bm{x}_p \in D_j^c$ and $\bm{x}_q \in D_k^c$. $\bm{a}_j^c$ and $\bm{a}_k^c$ are subgradients of $f^c$ at $\bm{x}_p$ and $\bm{x}_q$, respectively. By applying \cref{fact5}, it follows that $(\bm{x}_p \in D_j^c) \wedge (\bm{x}_q \in D_k^c) \Rightarrow (\bm{x}_p-\bm{x}_q)^T(\bm{a}_j^c - \bm{a}_k^c) \geq 0$. This can be expressed as \cref{eq_monoticity} using indicator variables $\delta_{p,j}^c$ and $\delta_{q,k}^c$ and an appropriately large value for the big-M parameter $M_{p,q}^c$. 
\end{proof}

\begin{theorem}
\label{theorem_bigM2}
    For well-behaved \gls{cpwl} solutions, the following value is a valid tightening value for the big-M parameter $M_{p,q}^c$:
    \begin{equation}
    \label{eq_bigM2}
        M_{p,q}^c = M_p^c + M_q^c
    \end{equation}
\end{theorem}

\begin{proof}[Proof of \cref{theorem_bigM2}]
    Let $p,q \in \{1,...,N\}$, $j,k \in \{1,...,P^c\}$. Owing to \cref{lemma5}, we have:
    \begin{gather*}
    (\bm{a}_j^c - \bm{a}_k^c)^T (\bm{x}_p-\bm{x}_q)
    = {\bm{a}_j^c}^T(\bm{x}_p-\bm{x}_q) - {\bm{a}_k^c}^T(\bm{x}_p-\bm{x}_q)\\
    = (f_j^c(\bm{x}_p) - f_j^c(\bm{x}_q)) - (f_k^c(\bm{x}_p) - f_k^c(\bm{x}_q))\\
    = (f_j^c(\bm{x}_p) - f_k^c(\bm{x}_p)) - (f_j^c(\bm{x}_q) - f_k^c(\bm{x}_q))
    \geq -M_p^c - M_q^c
    \end{gather*} 
\end{proof}

The number of additional constraints introduced by \cref{eq_monoticity} is $N(N-1)(P^++P^--2)/2$.

\subsection{Using the convexity of the affine domains}
Let $\Delta^n(i_1,...,i_{n+1})$ denote the $n$-simplex formed by the vertices $\bm{x}_{i_1},...,\bm{x}_{i_{n+1}} \in \mathbb{R}^d$, for all $n \leq d$. In particular, $\Delta^1(i_1,i_2)$ represents the line segment connecting points $\bm{x}_{i_1}$ and $\bm{x}_{i_2}$. Let $\{D_j\}_{j=1,...,P}$ be a partition of $D$ into convex regions.

\begin{lemma}
\label{lemma_point_in_simplex}
    If a point $\bm{x}$ belongs to the interior of a $d$-simplex $\Delta^d$, and all the vertices of $\Delta^d$ belong to the same convex region $D_j$, then $\bm{x}$ also belongs to $D_j$.
\end{lemma}

\begin{proof}[Proof of \cref{lemma_point_in_simplex}]
    By the definition of convexity, $\Delta^d$ is entirely contained within $D_j$. Therefore, $\bm{x}$ also belongs to $D_j$. 
\end{proof}

\begin{lemma}
\label{lemma_line_cross_face}
    Assume that a line segment $\Delta^1$ crosses a $(d-1)$-simplex $\Delta^{d-1}$, no vertex of $\Delta^1$ is a vertex of $\Delta^{d-1}$, and the vertices of $\Delta^{d-1}$ all belong to the same convex region $D_j$. Also assume that a vertex of $\Delta^1$ belongs to $D_k$, with $D_k \neq D_j$. Then, the two vertices of $\Delta^1$ belong to different convex regions.
\end{lemma}

\begin{proof}[Proof of \cref{lemma_line_cross_face}]
    Let $\bm{x}$ be the point of intersection of $\Delta^{d-1}$ and $\Delta^1$. $\bm{x}$ is distinct from the vertices of $\Delta^{d-1}$ or $\Delta^1$. We will prove by contradiction that the two vertices of $\Delta^1$ cannot belong to $D_k$. Assume that both vertices of $\Delta^1$ belong to $D_k$. By convexity of $D_k$, it follows that $\Delta^1 \subset D_k \Rightarrow \bm{x} \in D_k$. Similarly, since $\Delta^{d-1} \subset D_j \Rightarrow \bm{x} \in D_j$. As a result, $\bm{x} \in D_j \cap D_k$. Since $\{D_j\}_{j=1,...,P}$ is a partition of $D$, $Int(D_j) \cap Int(D_k) = \varnothing$. Thus, $\bm{x}$ must belong to the common boundary of $D_j$ and $D_k$, i.e. $\bm{x} \in \partial D_j \cap \partial D_k$. According to the supporting hyperplane theorem \cite{boyd_convex_2004}, there exists a hyperplane containing $\bm{x}$, with $D_j$ entirely contained in one of the two half-spaces bounded by the hyperplane. Since $\Delta^{d-1} \subset D_j$, the supporting hyperplane of $D_j$ must be the affine span of $\Delta^{d-1}$. This implies that $D_j$ can only reside on one side of the affine hyperplane generated by $\Delta^{d-1}$. Consequently, $\Delta^{d-1} \subset \partial D_j$. Thus, $D_k$ must be on the opposite side of the hyperplane. However, both vertices of $\Delta^1$ are assumed to belong to $D_k$ but must reside on opposite sides of the hyperplane due to $\Delta^1$ crossing $\Delta^{d-1}$. This results is a contradiction. Therefore, the two vertices of $\Delta^1$ must belong to different convex regions. 
\end{proof}

\begin{theorem}
\label{theorem_convex_domain}
    The following equations represent valid tightening constraints of $MILP1(Q)$:
    \begin{equation}
    \label{eq13}
        \begin{aligned}
            \sum_{n=1}^{d+1}{\delta_{p_n,j}^c} - d \leq \delta_{q,j}^c, \\ p_1,...,p_{d+1},q \in \{1,...,N\}: 
            \\ q \notin \{p_1,...,p_{d+1}\},
            \quad \bm{x}_q \in \Delta^d(p_1,...,p_{d+1}),
            \quad j \in \{1,...,P^c\}, \quad c \in \{+,-\}
        \end{aligned}
    \end{equation}
    \begin{equation}
    \label{eq14}
        \begin{aligned}
            \sum_{n=1}^{d}{\delta_{p_n,j}^c} - d + \delta_{q_1,k}^c + \delta_{q_2,k}^c \leq 1 , \\p_1,...,p_d,q_1,q_2 \in \{1,...,N\}: 
            \\ \{p_1,...,p_d\} \cap \{q_1,q_2\} = \varnothing,
            \quad \Delta^{d-1}(p_1,...,p_d) \cap \Delta^1(q_1,q_2) \neq \varnothing,
            \\ j,k \in \{1,...,P^c\}: \quad j \neq k, \quad c \in \{+,-\}
        \end{aligned}
    \end{equation}
\end{theorem}

\begin{proof}[Proof of \cref{theorem_convex_domain}]
    According to \cref{fact1} and \cref{fact2}, $\{D_j^c\}_{j=1,...,P^c}$ forms a partition of $D$ into convex regions. Therefore, we can apply \cref{lemma_point_in_simplex} and \cref{lemma_line_cross_face}. We have $\delta_{i,j}^c = 1 \Leftrightarrow \bm{x_i} \in D_j^c$ and $\Delta^d(p_1,...,p_{d+1}) \subset D_j^c \Leftrightarrow \sum_{n=1}^{d+1}{\delta_{p_n,j}^c} = d+1$. Therefore, ``$\Delta^d(p_1,...,p_{d+1}) \subset D_j^c \Rightarrow \bm{x}_q \in D_j^c$'' is equivalent to ``$\sum_{n=1}^{d+1}{\delta_{p_n,j}^c} - d = 1 \Rightarrow \delta_{q,j}^c=1$'', which can be expressed as: $\sum_{n=1}^{d+1}{\delta_{p_n,j}^c} - d \leq \delta_{q,j}^c$. Similarly, ``$(\Delta^{d-1}(p_1,...,p_d) \subset D_j^c) \wedge (\bm{x}_{q_1} \in D_k^c) \Rightarrow \bm{x}_{q_2} \notin D_k^c$'' is equivalent to ``$\sum_{n=1}^{d}{\delta_{p_n,j}^c} - d + \delta_{q_1,k}^c = 1 \Rightarrow \delta_{q_2,j}^c=0$'', which can be expressed as: $\sum_{n=1}^{d}{\delta_{p_n,j}^c} - d + \delta_{q_1,k}^c + \delta_{q_2,k}^c \leq 1$. 
\end{proof}

\Cref{fig_simplex_face} illustrates an example involving a combination of $d$-simplex and point considered in \cref{eq13}, as well as $(d-1)$-simplex and line segment considered in \cref{eq14} for the case $d=2$. \Cref{theorem_convex_domain} implies that if points 1, 2, and 3 all belong to the same affine domain, then so does point 4. In addition, if points 5 and 6 belong to the same affine domain, distinct from the affine domain point 7 belongs to, then points 7 and 8 must belong to distinct affine domains. The impact of these equations on the feasible region of $MILP1(Q)$ remains unclear, but they do not eliminate any \gls{cpwl} solutions, well-behaved or not.

\begin{figure}
	\centering 
	\includegraphics[width=0.8\textwidth]{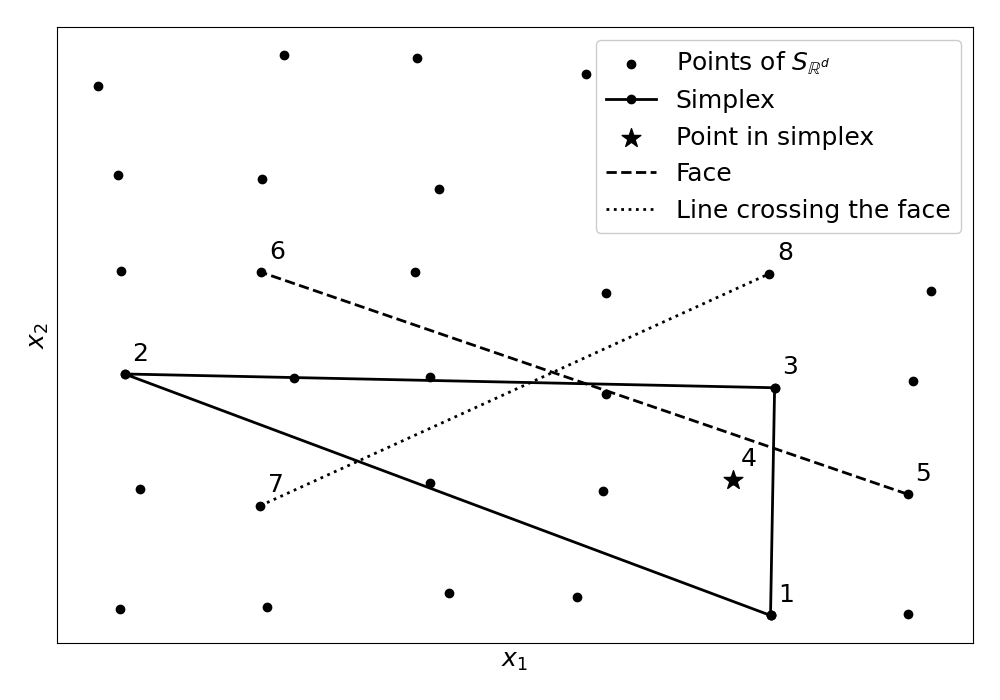}	
	\caption{Example in $d=2$ of a point (4) inside a 2-simplex (1,2,3) and a line segment (7,8) crossing a 1-simplex (5,6).} 
	\label{fig_simplex_face}
\end{figure}

\begin{proposition}
\label{proposition_number_simplex_csts}
    The number of constraints from \cref{eq13} is at most $\binom{N}{d+2}(d+2)(P^++P^-)$. The number of constraints from \cref{eq14} is at most $\binom{N}{d+2}(d+1)(d+2)\left(P^+(P^+-1) + P^-(P^--1)\right)/2$.
\end{proposition}

\begin{proof}[Proof of \cref{proposition_number_simplex_csts}]
    There are $\binom{N}{d+1}$ possible $d$-simplices and there are $(N-d-1)$ remaining points that may or may not belong to each simplex. Each constraint of \cref{eq13} applies to a single affine domain of $f^+$ or $f^-$. Furthermore, $\binom{N}{d+1}(N-d-1) = \binom{N}{d+2}(d+2)$. There are $\binom{N}{d}$ possible $(d-1)$-simplices and there are $\binom{N-d}{2}$ possible line segments that may or may not cross each $(d-1)$-simplex. Each constraint of \cref{eq14} applies to two distinct affine domains of $f^+$ or $f^-$. In addition, $\binom{N}{d}\binom{N-d}{2} = \binom{N}{d+2}\frac{(d+1)(d+2)}{2}$.
\end{proof}

\section{Proof of \cref{time_complexity}}
\label{proof_time_complexity}

The time complexity for calculating the bounds and big-M parameters is $O(2^{d+1}N^{d+2}d)$.

\begin{proof}[Proof of \cref{time_complexity}]
    To compute these values, we first need to identify all the affine functions of $A^*_{\varepsilon}(S)$. As seen in \cref{number_affine_functions}, the number of such affine functions is $\binom{N}{d+1}2^{d+1}$. As seen in \cref{fact_linear_interpolation}, for a given $s=(\bm{x}_{i_k},z_{i_k})_{k=1,...,d+1} \in [S]^{d+1}$, the linear coefficients $a_r$ and bias terms $b$ of the affine functions are calculated by solving the $2^{d+1}$ linear systems:
    $$
    \bm{M}_s 
    \left(
    \begin{array}{c}
        a_1 \\
        \vdots \\
        a_d \\
        b 
        \end{array}
    \right)
    = \bm{z}_s + \bm{e}
    $$
    where
    $\bm{M}_s = \left(
        \begin{array}{cccc}
        x_{i_1,1}
        & \cdots 
        & x_{i_1,d} 
        & 1  \\
        \vdots 
        & \ddots 
        & \vdots 
        & \vdots\\
        x_{i_{d+1},1}  
        & \cdots 
        & x_{i_{d+1},d} 
        & 1
        \end{array}
    \right) $,
    % $X=\left(
    % \begin{array}{c}
    %     a_1 \\
    %     \vdots \\
    %     a_d \\
    %     b 
    %     \end{array}
    % \right)
    % $,
    $\bm{z}_s=\left(
    \begin{array}{c}
        z_{i_1} \\
        \vdots \\
        z_{i_{d+1}} 
        \end{array}
    \right)$, and $\bm{e} \in \{-\varepsilon,\varepsilon\}^{d+1}$.
    That is, the affine functions are calculated by computing $\bm{M}_s^{-1}(\bm{z}_s + \bm{e})$.
    As a first step, we need to calculate the inverse $\bm{M}_s^{-1}$ for $\binom{N}{d+1}$ matrices in $M_{d+1}(\mathbb{R})$. A matrix inversion has a time complexity of $O((d+1)^3)$, and $\binom{N}{d+1} = O(N^{d+1})$. Thus, the total number of operations in that step is $O(N^{d+1}(d+1)^3)$. As a second step, for each of the $\binom{N}{d+1}$ $s \in [S]^{d+1}$, the coefficients $a_r$ and $b$ are calculated by computing $\bm{M}_s^{-1}(\bm{z}_s + \bm{e})$ for $2^{d+1}$ vectors $\bm{z}_s + \bm{e}$. The total time complexity for that step is $O((2N)^{d+1}(d+1)^2)$. Assuming $N \gg d$, the time complexity for the second step dominates that of the first step, and the overall time complexity of identifying all affine functions in $A^*_{\varepsilon}(S)$ is $O((2N)^{d+1}(d+1)^2)$.
    Calculating $\underline{a_r}$, $\overline{a_r}$, $\underline{b}$, $\overline{b}$ require calculating $2(d+1)$ extrema over a range of $\binom{N}{d+1}2^{d+1}$ values, which has a total complexity of $O((2N)^{d+1}(d+1))$. The time complexity of calculating $\overline{a'_r}$ and $\overline{b'}$ is only $O(d+1)$. For a given $i \in \{1,...,N\}$, computing $M_i^-$ requires evaluating the value of each of the $\binom{N}{d+1}2^{d+1}$ affine functions at the point $\bm{x}_i$, which requires $d+1$ operations for each affine function, and then calculate the extrema over the range of values. Because it has to be done for each point, the time complexity of calculating the $M_i^-$ is $O(N(2N)^{d+1}(d+1))$, which dominates the procedure of calculating all affine functions. Note that calculating the $M_i^+$ only requires $O(N)$ additional operations since $M_i^+ = \frac{\min(P^+-1,P^-)}{\min(P^--1,P^+)}M_i^-$. Finally, once the $M_i^c$ have been calculated, calculating $M_{p,q}^c = M_p^c + M_q^c$ only requires $O(N^2)$ operations. 

\end{proof}

\end{document}